\documentclass{amsart}
\usepackage{amsfonts,amscd, amssymb}
\usepackage{tikz-cd}
\newtheorem{theorem}{Theorem}[section]
\newtheorem{lemma}[theorem]{Lemma}
\newtheorem{corollary}[theorem]{Corollary}
\newtheorem{proposition}[theorem]{Proposition}

\theoremstyle{remark}
\newtheorem{remark}[theorem]{Remark}
\theoremstyle{definition}
\newtheorem{definition}[theorem]{Definition}
\newtheorem{example}[theorem]{Example}
\numberwithin{equation}{section}
\makeatother

\DeclareMathOperator{\bB}{{\mathbb B}}

\DeclareMathOperator{\bR}{{\mathbb R}}
\DeclareMathOperator{\bC}{{\mathbb C}}
\DeclareMathOperator{\bH}{{\mathbb H}}
\DeclareMathOperator{\bN}{{\mathbb N}}

\DeclareMathOperator{\Rdb}{{\mathbb R}}

\DeclareMathOperator{\maxten}{\otimes_{\rm max}}
\DeclareMathOperator{\minten}{\otimes_{\rm min}}

\DeclareMathOperator{\cR}{{\mathcal R}}
\DeclareMathOperator{\cS}{{\mathcal S}}

\DeclareMathOperator{\decR}{{{\rm Dec}_{\mathbb{R}}}}
\DeclareMathOperator{\decC}{{{\rm Dec}_{\mathbb{C}}}}
\DeclareMathOperator{\idec}{{\rm dec}}
\DeclareMathOperator{\cpR}{{{\rm CP}_{\mathbb{R}}}}
\DeclareMathOperator{\cpC}{{{\rm CP}_{\mathbb{C}}}}

\DeclareMathOperator{\cbR}{{{\rm CB}_{\mathbb{R}}}}
\DeclareMathOperator{\cbC}{{{\rm CB}_{\mathbb{C}}}}
\DeclareMathOperator{\icb}{{\rm cb}}

\DeclareMathOperator{\scp}{{\rm SCP}}

\usepackage{tikz}

\usepackage{pst-node}
\usepackage{tikz-cd} 

\usepackage[hidelinks]{hyperref}

\title{Real decomposable maps on operator systems}

\author{David P. Blecher}
	
\address{Department of Mathematics, University of Houston, Houston, TX 77204-3008.}
	
\email{dpbleche@central.uh.edu}

\author{Christiaan H. Pretorius}

\address{Department of Mathematics, University of Houston, Houston, TX 77204-3008.}
	
\email{chpretor@central.uh.edu} 

\date{May 6, 2026, final version} 

\begin{document}

\begin{abstract}   We initiate and study the theory of ``real decomposable maps" between real operator systems.  Formally, this is new even in the complex case, which hitherto has restricted itself to the case where the systems are complex $C^*$-algebras.   We investigate how our definition interacts with the existing theory (which it generalizes) and with the complexification.  In particular, a surprising term appears in the `Jordan decomposition' of real decomposable  maps.  This term constitutes  a new class of completely bounded maps, a class that also showed up  in disguised form in our recent study of real noncommutative (nc) convexity, and  whose theory is likely to have applications in that subject.  We also check the real case of many important known results related to decomposability, for example results about the weak expectation property or injectivity of von Neumann algebras. 
  \end{abstract}

\maketitle

\section{Introduction}  The decomposable maps may be viewed as a very useful class between the completely positive maps and the completely bounded maps.  They 
have some of the strong properties of completely positive maps but are much more general.  Indeed, 
analogously to the Jordan decomposition, the set of (complex) decomposable maps  are initially defined as the maps belonging to the complex span of the completely positive maps.  Haagerup \cite{Haag} and Pisier \cite{Pisbk,P} developed the theory of (complex) decomposable maps between $C^*$-algebras, however most of their results and their proofs clearly work more generally for maps between operator systems $V$ and $W$ (c.f.\ \cite{Han}). 
Thus $$\decC(V,W) = {\rm Span}_{\bC}(\cpC(V,W)) \subseteq {\rm CB}_{\bC}(V,W)$$ for complex operator systems, where these 
are respectively the decomposable,  completely positive (cp) maps, and  completely bounded maps.  We shall see however that this relation is not correct in the real case, in general.   Even in the complex case this relation hides the appropriate norm 
on $\decC(V,W)$.  From Haagerup's seminal work on decomposable maps \cite{Haag}, we have that $u: V\to W$ belongs to $\decC(V,W)$ if and only if there exist completely positive maps $S_1,S_2: V\to W$ such that
\begin{equation}\label{eq:OGCP}
x \mapsto\begin{bmatrix}
    S_1(x) & u(x)\\
    u(x^*)^* & S_2(x)
\end{bmatrix}
\end{equation}
defines a completely positive map from $V$ to $M_2(W).$ (Of course $V, W$ were $C^*$-algebras in \cite{Haag}.) Haagerup also defined the accompanying norm $\|\cdot\|_{\idec}$ on $\decC(V,W)$ by $\|u\|_{\idec} = \inf\{\max \{\|S_1\|,\|S_2\|\}\},$ where the infimum runs over all $S_1,S_2\in \cpC(V,W)$ such that the expression in \hyperref[eq:OGCP]{(\ref*{eq:OGCP})} yields a completely positive map. 
This norm is just the usual completely bounded norm and $\decC(V,W) = {\rm CB}(V,W)$ if $W$ is an injective $C^*$-algebra.
 The Banach space $(\decC(V,W),\|\cdot\|_{\idec})$ has proved to be, and will continue to be, a valuable tool in operator algebras. Decomposable maps  have been used to study the Connes-Kirchberg problems.
A representative example of a very recent paper using decomposable maps  to study completely bounded norms of k-positive maps in quantum information theory is  \cite{APet}, and we understand that some of these authors are continuing this direction
(see e.g.\ the last paragraphs in Section 1 in \cite{DPR}).  See \cite{P} for a modern survey of decomposable maps between complex $C^*$-algebras with applications. 

We shall see that it is precisely decomposability in the sense of \hyperref[eq:OGCP]{(\ref*{eq:OGCP})} that is the correct definition of a real decomposable map.  That is, $u$ is 
a real decomposable map between real operator systems, if there exist real completely positive maps $S_1,S_2,$ such that the prescription in \hyperref[eq:OGCP]{(\ref*{eq:OGCP})} defines a completely positive map. We will denote the space of such maps between real operator systems $V$ and $W,$ together with Haagerup's norm above, by $\decR(V,W).$

Here we investigate the theory of  ``real decomposable maps", generalizing (and relating it to) the existing complex theory, mostly via the complexification.  An important point is that the real case is a strict generalization of the complex case.  This is a subtle point: it is obvious that complex operator systems are real complex operator systems, however it is true, but not at all obvious, that a complex map between complex operator systems is real decomposable if and only if it is complex decomposable.  
Our paper is a sequel to \cite{BR}, which is a systematic development of the theory of real operator systems, but which skipped over topics connected to decomposability.  In turn \cite{BR} was a sequel to \cite{BReal}, which systematically develops real operator space theory. As was the case in these papers, many results will follow from the functoriality of complexification.  That is,  loosely speaking, that $F(X_c)\cong F(X)_c$ for a real object $X$ and various constructions $F.$  Here $X_c$ denotes the complexification of $X$.  Of course many 
results follow by the same proof in the complex case.
Nonetheless we consistently provide  `complexification proofs' where these are available.  This is partly because these are usually much shorter, but also 
because sometimes subtle errors occur in converting a complex proof to the real case, such as mistakenly (forgetfully) assuming that a certain real map is complex linear.

In \hyperref[thm:RealIsComplex_OS]{Theorem \ref*{thm:RealIsComplex_OS}} in \hyperref[sec:RealDecMaps]{Section \ref*{sec:RealDecMaps}}, we prove that the complexification of $\decR(V,W)$ is the set of 
(complex) decomposable maps $\decC(V_c,W_c)$.
The span of the completely positive maps equals the 
set of selfadjoint decomposable maps $\decR(V,W)_{\rm sa}$  (that is, decomposable maps for which $u(x) = u(x^*)^*$).  We have a somewhat surprising `Jordan decomposition'
$$\decR(V,W) = 
\cpR(V, W) - \cpR(V, W)  + \decR(V,W)_{\rm as},$$ where $\decR(V,W)_{\rm as}$ denotes the set of {\em antisymmetric} or {\em skew} decomposable maps, namely those for which $u(x^*)^* = -u(x)$.  This was 
surprising to us in two ways.
First, one might have expected  
in the Jordan decomposition in the real case just $\cpR(V, W) - \cpR(V, W)$, as in the classical Jordan decomposition.  Second,  the 
extra antisymmetric term turns out to be precisely an important class in the first author's  recent real noncommutative (nc) convexity work \cite{BMcI} inspired by \cite{DK}, and a special case of it shows up in the complexification of the nc state space.  In this theory 
the complexification of a real nc convex set $K$ (and in particular of the real nc state space) is a set of elements $x + iy$ for certain 
$x$ and $y$, the `real' and `imaginary' parts.
The `real part' $x$ is in $K$ and is easy to understand and work with, but $y$ has usually been 
more difficult to comprehend or to get 
one's hands on explicitly.  However such elements $y$ turn out to essentially  correspond to the antisymmetric term in the Jordan decomposition above, as we shall see in 
Theorem \ref{coco}.  
 This suggests that it should be useful to view and treat the imaginary parts of complex nc convex sets with tools from the theory of decomposable maps. Such antisymmetric terms give important new examples of decomposable maps in the real case.  They are also the `imaginary part'  of a complex completely positive map $\varphi$ between complexifications of operator systems.  See Example \ref{exrd} and Section \ref*{sec:DecAs}.
Thus these maps are quite ubiquitous, since it turns out that it is extremely difficult to find explicit complex operator systems that are not complexifications.   (For $C^*$-algebras this is equivalent to the hard problem of the existence of a real form. There are not even any known simple low dimensional examples of complex Banach spaces that are not (isometric) complexifications.) In  \hyperref[sec:deltaNorm]{Section \ref*{sec:DecAs}} we characterize $\decR(V,W)_{\rm as}$ in various ways, and  give a Stinespring-like theorem for such maps.

In \hyperref[sec:Apps]{Section \ref*{sec:Apps}} the real case of various tensor-related applications of real decomposable maps from \cite{P, Haag} are discussed; such as the real version of the decomposable characterizations of WEP, QWEP, and of injective von Neumann algebras.  We also check the real case of some other related results from \cite{P}.  We do not attempt to cover every result,  however we cover enough so that the reader familiar with the complex theory from the literature will obtain  a fairly complete grasp of how the real case works out.   We remark that the first author has some earlier work with Christian Le Merdy on a notion of  decomposable maps on nonselfadjoint operator algebras.  

We will expect our readers to be a little familiar with 
the theory of real operator systems and operator spaces (see e.g.\ \cite{BR,BReal}).  Indeed, some of our proofs here are more condensed than similar proofs in e.g.\ \cite{BR,BMcI} since more familiarity with some of the tricks and techniques there is assumed. 
        For general background on complex operator systems and spaces, and in particular on the definitions etc.\ in the rest of this section, we refer the reader to e.g.\ \cite{Pnbook,BLM,Pisbk}.  
It might 
be helpful to also browse some of the other  existing real operator space theory  e.g.\ \cite{ROnr,RComp,
BT,BCK}.  Some basic 
real $C^*$- and von Neumann algebra theory may be found in \cite{Li}.  

We write $M_n(\bR)$ for the real $n \times n$ matrices, or sometimes simply $M_n$ when the context is clear.  Similarly in the complex case.
If $n$ is an infinite cardinal, we 
interpret $M_n(X)$ in the usual 
operator space sense (see 1.2.26 in \cite{BLM} for the complex case).  In particular, 
$M_n(\bR) \cong B(l^2_n(\bR))$. 
If $X = Y^*$ is a dual 
operator space, then $M_n(X)$ may be identified with ${\rm CB}(Y,M_n)$ (see \cite[Section 1.6]{BLM}
for the complex case).  The above hold for any cardinal $n$. 
 The letters $H, K$ are usually reserved for real or complex Hilbert spaces.  Every complex Hilbert space $H$ is a real Hilbert space, i.e.\ we forget the complex structure.  
 We write $X_{\rm sa}$ for the selfadjoint elements  
 in a $*$-vector space $X$.  In the complex case 
 $M_n(X)_{\rm sa} \cong (M_n)_{\rm sa} \otimes X_{\rm sa}$, but this  fails for real spaces. A subspace of $B(H)$ is {\em unital} if it contains the identity,   
and a map $T$ is unital if $T(1) = 1$.  Our identities $1$ always have norm $1$.

Clearly a real theory of  decomposable maps must involve real operator systems.
 A {\em concrete complex (resp.\ real) operator system} $V$ is a unital selfadjoint subspace of $B(H)$ for $H$ a complex (resp.\ real) Hilbert space. For $n \in \bN$ we have the identification $M_n(B(H)) \cong B(H^{(n)})$ where $H^{(n)}$ is the $n$-fold direct sum of $H$. From this identification, $M_n(V)$ inherits a norm and positive cone $\mathfrak{C}_n$. The latter is the set $M_n(V)^+ = \{x \in M_n(V): x = x^* \geq 0 $ in $ B(H^{(n)})\}$.  
 
 For  $n \in \bN$ we define the amplification of a linear map $\varphi: V \to W$ by 
        \[\varphi^{(n)}: M_n(V) \to M_n(W)\]
        \[[x_{ij}] \mapsto [\varphi(x_{ij})] . \]
        The natural morphisms between operator systems are {\em  unital completely positive} (ucp) maps, which are linear maps $\varphi: V \to W$ that are unital and every amplification is positive (or equivalently selfadjoint and contractive).  The norm of a cp map $\varphi$ agrees with its completely
        bounded (cb) norm $\| \varphi \|_{\rm cb} = \sup_n \| \varphi^{(n)} \|$, and also equals
        $\| \varphi(1) \|$.
        The isomorphisms (resp.\ embeddings) of operator systems which are used in this paper are bijective (resp.\ injective) ucp maps whose inverse (resp.\ inverse in its range) is ucp.  These are called unital complete order isomorphisms (resp.\ unital complete order embeddings); or {\rm ucoi} (resp.\ ucoe) for short. They are necessarily completely isometric. 
        An abstract real operator system is a 
        real $*$-vector space $V$ with a fixed `unit element' $1$, and cones $\mathfrak{C}_n \subset M_n(V)$, 
        which is ucoi to a concrete operator system.
        See also e.g.\  \cite[Section 2]{BR}. 

        Similarly a concrete operator space $E$ is a subspace of $B(H)$ with norms on $M_n(E)$ inherited from $B(H^{(n)})$. There are abstract characterizations of real/complex operator spaces and operator systems 
        which we will not repeat here.
 If $V$ is a real operator system then $V$ can naturally be made into a complex operator system by complexification. The complexification $V_c$ of $V$ is algebraically just the vector space complexification of $V$, consisting of elements $x+iy$ for $x,y \in V$. We give this a conjugate linear involution $(x+iy)^* = x^* - i y^*$. The matrix ordering $M_n(V_c)^+$ will be defined by
 \begin{equation}\label{mo}       M_n(V_c)^+ = \{x+iy \in M_n(V_c): c(x,y) \geq 0\} , \end{equation}
        where $c(x,y)$ is the matrix in $M_{2n}(V)$ defined by 
        \begin{equation}\label{cis} c(x,y) = \begin{bmatrix}
            x & -y\\ y & x
        \end{bmatrix} . \end{equation}
   Then $V_c$ is an abstract operator system, called the operator system  complexification  of $V$.
   We recall that this complexification  is the unique one satisfying Ruan's completely {\em reasonable} condition, namely that the map $\theta_V(x+iy)=x-iy$ is a ucoi. 
That is, as described in and after (2.1) in \cite{BCK}, we may identify
$V_c$ unitally real completely order isomorphically (and hence completely isometrically) with the real operator system 
$$\cR_V = \{ c(x,y) \in M_2(V) : x ,y \in V \}.$$
There is a natural `multiplication by i' on the latter, with respect to which 
$c : V_c \to \cR_V$ becomes a complex complete order isomorphism.  Thus often in our paper we simply work 
with $\cR_V$ in place of $V_c$ without comment.  Similarly for a 
real linear $u : V \to W$, the complexification 
$u_c$ may be identified with $(u_2)_{|\cR_V}$.  We will assume that the reader is familiar with basic properties of $u_c$ from e.g.\ early parts of \cite{BR}.
We recall that $\theta_V := x+iy\mapsto x -iy,\ x,y\in V$ is a period 2 real complete order automorphism of $V_c$.   Conversely, the set $V$ of fixed points of any period 2 real complete order automorphism of a complex operator system $W$,  is a real operator system with $V_c = W$. 

We will use the matrix notation $c(x,y)$ very often in our paper.  Sometimes we also write $c(x+iy)$ for $c(x,y)$.  Note that for a real $C^*$-algebra $A$, and in particular for $A = M_m(\bR)$ the map $c : A_c \to M_2(A)$ is a  faithful $*$-homomorphism. 

        In parts of \cite{BR,BMcI} it was checked that many of the basic theorems and constructions for complex operator systems also hold for real operator systems. Very many foundational
structural results for real operator systems were developed, and it was shown  how the complexification interacts
with the basic constructions in the subject.

\section{Real decomposable maps}\label{sec:RealDecMaps}

\begin{definition}\label{def:DecR}
Let $V$ and $W$ be real operator systems. We will say a linear map $u: V\to W$ is {\em real decomposable}, or just {\em decomposable}, if there exists $S_1,S_2\in \cpR(V, W)$ (the set of real cp maps between $V$ and $W$) such that
\begin{equation}\label{decd}
\hat{u} = \begin{bmatrix}
S_1 & u\\
u^* & S_2
\end{bmatrix} \in \cpR(V,M_2(W)),
\end{equation} where $u^*(x) := u(x^*)^*.$ Let us write $\mathfrak{P}(u)$ for the set of all pairs $(S_1,S_2)$ that satisfy the prescription above. These are the {\em witnesses of decomposability}. We will write $\decR(V,W)$ for the set of all maps that are real decomposable. We also define $\|\cdot\|_{\idec}: \decR(V,W) \to \Rdb,$ by $$\|u\|_{\idec}= \inf\{\max\{\|S_1\|, \|S_2\|\} : (S_1,S_2) \in \mathfrak{P}(u)\}.$$
\end{definition}

One may view $\hat{u}$ as somewhat analogous to the variation (measure) of a (nonpositive) measure.

\begin{proposition}\label{ss}  In the above definition, in both the real and the complex case, one may choose $S_1, S_2$  to be ucp if $\|u\|_{\idec} < 1$.  For such a choice,  $\hat{u}$ is ucp if and only if $u(1) = 0$.
\end{proposition}

\begin{proof} Choose $(S_1,S_2) \in \mathfrak{P}(u)$ 
with $\| S_k \| = \| b_k \| \leq 1$, where $b_k = S_k(1)$, and let $\hat{u}$ be the associated cp map from (\ref{decd}). Let $\psi$ be any state of 
$V$, and let $R_k = \psi(\cdot) \, (I-b_k) \in {\rm CP}(V,W)$. Define $\varphi = \hat{u} +  (R_1 \oplus R_2)$.  This is cp, so that $(S_1 + R_1, S_2+R_2) \in \mathfrak{P}(u)$.   Replacing $S_k$ by $S_k+R_k$ and $\hat{u}$ by $\varphi$ establishes the first result. The second is obvious. 
\end{proof} 

{\bf Remark.} We shall see later (in Remark \ref{att}) that the proposition is also valid if $\|u\|_{\idec} = 1$ under further reasonable conditions on $W$.  We do not know whether these conditions are necessary.

\bigskip 

Unless there is cause for confusion, we will use 
$\|\cdot\|_{\idec}$ for the decomposable norm in both the real and complex case;  the reader should easily be able to make sense of what is being communicated.

\begin{theorem}\label{res:RealInComplex}
Let $V,W$ be real operator systems. Then $u\in \decR(V,W)$ if and only if $ u_c\in \decC(V_c,W_c).$ Moreover $\|u\|_{\idec} = \|u_c\|_{\idec}.$
\end{theorem}

\begin{proof}
Let $u\in \decR(V,W)$.  Note that  $(u^*)_c = (u_c)^*.$ (See \hyperref[def:DecR]{Definition \ref*{def:DecR}} for the definition of $u^*.$) Indeed, let $x,y\in V,$ then 
$$\begin{aligned}
\begin{bmatrix}
    u^*(x) & u^*(-y)\\
    u^*(y) & u^*(x)
\end{bmatrix} =
\begin{bmatrix}
    u(x^*) & u(y^*)\\
    u(-y^*) & u(x^*)
\end{bmatrix}^*
= \left(u_c\left(\begin{bmatrix}
    x^*& y^*\\
    -y^* & x^*
\end{bmatrix}\right)\right)^*
\end{aligned}.$$
This shows that $(u^*)_c = (u_c)^*$, since under the 
identifications for the complexification  discussed towards the end of the introduction,  $$(u^*)_c\left(\begin{bmatrix}
    x& -y\\
    y& x
\end{bmatrix}\right) = (u_c)^*\left(\begin{bmatrix}
    x& -y\\
    y & x
\end{bmatrix}\right).$$

Suppose that $u\in \decR(V,W),$ and $(S_1,S_2) \in \mathfrak{P}(u).$ We let
$$\hat{u} := \begin{bmatrix}
    S_1 & u\\
    u^*& S_2
\end{bmatrix} \quad \text{and} \quad \vec{z} = c(x,y) = \begin{bmatrix}
    x & -y\\ y & x
\end{bmatrix}, \; \; \text{and}$$
$$
\begin{bmatrix}
    S_1(x) & u(x) & S_1(-y) & u(-y)\\
    u^*(x) & S_2(x) & u^*(-y) & S_2(-y)\\
    S_1(y) & u(y) & S_1(x) & u(x)\\
    u^*(y) & S_2(y) & u^*(x) & S_2(x)
\end{bmatrix} \sim
\begin{bmatrix}
    S_1(x) & S_1(-y) & u(x) & u(-y)\\
    S_1(y) & S_1(x) & u(y) & u(x)\\
    u^*(x) & u^*(-y) & S_2(x) & S_2(-y)\\
    u^*(y) & u^*(x) & S_2(y) & S_2(x)
\end{bmatrix}
$$
where $\sim$ indicates  a ``canonical shuffle," which does not affect a map being cp. Note that the left matrix may be identified with $(\hat{u})_c(\vec{z})$ and the right with $[T_{ij}]$, where $$T_{11} = (S_1)_c(\vec{z}), 
\; T_{12} = u_c(\vec{z}), \; T_{21} = (u^*)_c(\vec{z}), \; T_{22} = (S_2)_c(\vec{z}).$$ Recall that 
$$(\hat{u})_c\left(\vec{z}\right) 
= \begin{bmatrix}
    (S_1)_c\left(\vec{z}\right)
&
u_c\left(\vec{z}\right)\\
(u_c)^*\left(\vec{z}\right)
&
(S_2)_c\left(\vec{z}\right)
\end{bmatrix}
= \widehat{u_c}\left(\vec{z}\right).$$
By compression with $\begin{bmatrix}
I_{M_2}\\ \mathbf{0}
\end{bmatrix}$, with $\mathbf{0}$ the zero matrix, we find that both $(S_1)_c$ and $(S_2)_c$ are positive. We have shown that if $u\in \decR(V,W),$ then $u_c\in \decC(V_c,W_c).$ Furthermore, we have 
$$\begin{aligned}
    \|u_c\|_{\idec} = \inf\{\max\{\|R_1\|,\|R_2\|\}: (R_1,R_2)\in \mathfrak{P}(u_c)\}\leq \max\{\|(S_1)_c\|,\|(S_2)_c\|\},
\end{aligned}$$
since $((S_1)_c,(S_2)_c)\in \mathfrak{P}(u_c).$ Furthermore $\|(S_i)_c\|= \|S_i\| $ for $i=1,2,$ hence $ \|u_c\|_{\rm dec} \leq \|u\|_{\rm dec}.$ 

For the converse direction, suppose that $v = u_c \in \decC(V_c,W_c)$ 
and let $(R_1,R_2)\in \mathfrak{P}(u_c).$ 
Let $\rho : W_c \to W$ be the `real part projection' 
$\rho(x+iy) = x$, which is ucp \cite{BR}. Let $S_i = \rho \circ (R_i)_{|V} \in \cpR(V,W)$.
We let
$$\hat{u} := \begin{bmatrix}
    S_1 & u\\
    u^*& S_2
\end{bmatrix} \quad \text{and} \quad \hat{v} := \begin{bmatrix}
    R_1 & v\\
    v^*& R_2
\end{bmatrix}  .$$ 
Then $\hat{v}$ is cp and hence so is $\rho^{(2)} \circ 
\hat{v}_{| V} = \hat{u}$.  Thus 
$(S_1,S_2)\in \mathfrak{P}(u)$. This shows that $u\in \decR(V,W),$ and 
$$\|u\|_{\rm dec} \leq \max \{ \| S_1 \| , \| S_2 \| \} \leq 
\max \{ \| R_1 \| \| R_2 \| \}.$$ 
Thus   $\|u_c\|_{\rm dec} \geq \|u\|_{\rm dec}.$ 
\end{proof}

Let $V,W$ be real operator systems. We will frequently make use of the fact that the space $(\cbR(V,W),\|\cdot\|_{\icb})$ embeds into $(\cbC(V_c,W_c),\|\cdot\|_{\icb})$ isometrically via $T\mapsto T_c.$ For more information on this result and these maps, we refer the reader to \cite[Proposition 1.1]{BReal}.

\begin{corollary}\label{prop:cb<dec}
 We have $\decR(V, W) \subseteq \cbR(V, W)$; furthermore, for all $u \in \decR(V, W)$ we have $\|u\|_{\icb} \leq\|u\|_{\rm dec}.$
\end{corollary}

\begin{proof} Via complexification and \cite[Proposition 6.6 (i)]{P} we have $\|u\|_{\rm cb} = \|u_c\|_{\rm cb} \leq \|u_c\|_{\idec} = \|u\|_{\idec}.$ 
\end{proof}

\begin{corollary}
Let $V,W$ be real operator systems. Then $\decR(V,W)$ is a Banach space, 
identifiable with a closed subspace of
$\decC(V_c,W_c).$
\end{corollary}

\begin{proof}
Let $u,v\in \decR(V,W)$ with $(S_1,S_2)\in \mathfrak{P}(u)$ and $ (T_1,T_2)\in \mathfrak{P}(v),$ and fix $\lambda \in \mathbb{R}.$ We have 
$$\begin{bmatrix}
   1& 0\\0&\lambda
\end{bmatrix}^*\begin{bmatrix}
    S_1& u\\u^*& S_2
\end{bmatrix}\begin{bmatrix}
    1& 0\\0&\lambda
\end{bmatrix} = \begin{bmatrix}
    S_1 & \lambda u\\ \lambda u^* & \lambda^2S_2
\end{bmatrix} \in \cpR(V,M_2(W)).$$
And since $\cpR(V,M_2(W))$ is a positive cone, we have
$$
\begin{bmatrix}
 S_1 & u\\
 u^* & S_2
\end{bmatrix}
+
\begin{bmatrix}
 T_1 & v\\
 v^* & T_2
\end{bmatrix}
=
\begin{bmatrix}
 S_1+T_1 & u+v\\
 (u+v)^* & S_2+T_2
\end{bmatrix}
\in \cpR(V,M_2(W)).$$ Hence $\decR(V,W)$ is a real vector space.   Clearly the map $u \mapsto u_c$ isometrically embeds $\decR(V,W)$ as a subspace of
$\decC(V_c,W_c)$ by \hyperref[res:RealInComplex]{Theorem  \ref*{res:RealInComplex}} and its proof. We claim that this real subspace of $\decC(V_c,W_c)$ is closed.  Let $u_n \in \decR(V,W)$, such that $(u_n)_c \to S \in (\decC(V_c,W_c),\|\cdot\|_{\idec}).$
We must show $S=u_c$ for some $u\in \decR(V,W).$ By \hyperref[prop:cb<dec]{Proposition \ref*{prop:cb<dec}} we have $(u_n)_c \to S$ in $\|\cdot\|_{\icb}.$ Since the copy of $\cbR(V,W)$ is closed in $\cbC(V_c,W_c)$ we see that  $S= u_c$ for $u \in \cbR(V,W)$.
By \hyperref[res:RealInComplex]{Theorem  \ref*{res:RealInComplex}} we have $u \in \decR(V,W)$.
 Since $\decC(V_c,W_c)$ is a Banach space (see \cite[Lemma 6.2]{P}), we are done.
\end{proof}

We will need  to discuss how the complexification $W + iW$ of a real (matrix) ordered space $(W,({\mathfrak C}_n))$ is (matrix) ordered.
As in (\ref{mo})  and in \cite{BMcI} we define matrix cones $\widetilde{{\mathfrak C}}_n \subseteq M_n(W + iW)$  by: $x + iy \in \widetilde{{\mathfrak C}}_n$ if and only if $c(x,y) \in {\mathfrak C}_{2n}$.  
The following result, most of whose details are mentioned in \cite{BMcI}, is consistent with  the operator system identity ncS$(V_c) \cong {\rm ncS}(V)_c$ in \cite[Lemma 3.7]{BMcI}, where ${\rm ncS}(V)$ is the nc state space, that is
$({\rm UCP}(V,M_n))$.  

\begin{lemma} \label{duofs}  For a real operator system $V$, and cardinal $n$, on $M_n((V_c)^*) = {\rm CB}(V_c , M_n(\bC))$  the canonical cone $\widetilde{{\mathfrak C}}_n$ above coincides with 
the expected cone ${\rm CP}(V_c , M_n(\bC))$.  Thus if 
$$\varphi \in M_n((V_c)^*) \cong {\rm CB}(V_c, M_n(\bC)) = {\rm CB}(V,M_n(\bR))_c,$$ so that (uniquely) 
$\varphi = \psi_c + i \sigma_c$ for $\psi, \sigma \in {\rm CB}(V,M_n(\bR))$, then  we have 
$$\varphi \in M_n((V_c)^*)^+ = {\rm CP}(V_c,M_n(\bC))$$ if and only if $c(\varphi)  = c(\psi, \sigma) \in 
M_{2n}(V^*)^+ = {\rm CP}(V,M_{2n}(\bR))$.   
(Here $c(\psi, \sigma)$ is as usual, see {\rm (\ref{cis})} for this notation.)
In this case $\| \varphi \| = \| c(\psi, \sigma) \|$, and $\sigma$ is skew (that is $\sigma^* = - \sigma$). \end{lemma}

       \begin{proof}  This follows by a variation of the proof of \cite[Lemma 3.7]{BMcI}.  Alternatively, it may be derived 
from the statement of  \cite[Lemma 3.7]{BMcI}, and \cite[Lemma 6.3]{BMcI}.   Indeed, suppose that 
$\varphi \in M_n((V_c)^*) \cong {\rm CB}(V_c, M_n(\bC))$, and (uniquely) $\varphi = \psi_c + i \sigma_c$ for $\psi, \sigma
\in {\rm CB}(V,M_n(\bR))$.  Then  one direction of the `if and only if'  is obvious since $\varphi = u_n^* c(\varphi) u_n$ where 
$u_n$ is as in \cite{BMcI}.  Thus $\| \varphi \| \leq \| c(\varphi) \|$. Since $$\varphi(x^*) = \psi(x^*) + i \sigma(x^*) = 
\varphi(x)^* = \psi(x)^* - i \sigma(x)^*, \qquad x \in V,$$ $\sigma$ is skew. 

Conversely, if $\varphi \in {\rm CP}(V_c,M_n(\bC))$ then by \cite[Lemma 6.3]{BMcI} we may write 
$\varphi  = \alpha \psi \alpha$ for $\alpha = \varphi(1)^{\frac{1}{2}} \in M_n(\bC)^+$ and $\psi \in {\rm UCP}(V_c,M_n(\bC))$.  Then $c(\psi) \in {\rm UCP}(V,M_{2n}(\bR))$,
and $c(\varphi)  =  c(\alpha) c( \psi) c( \alpha) \in {\rm CP}(V,M_{2n}(\bR))$. 
Thus  $$\| c(\varphi) \| \leq \| c(\alpha)\|^2 =   \| \alpha \|^2
= \| \varphi(1)^{\frac{1}{2}} \|^2 =  \| \varphi(1) \| 
= \|\varphi \|$$ as desired. 
  \end{proof}  

We call $\psi$ (resp.\ $\sigma$) the {\em real part} (resp.\ {\em imaginary part})  of $\varphi$.   More generally we use the same notation if 
$\varphi : V_c \to W_c$ is a complex completely positive map 
and $\varphi = \psi_c + i \sigma_c$ with $\psi, \sigma : V \to W$. 

 \begin{example} \label{exrd}   Let $V, W$ be real operator systems.
 The `imaginary part'  $\sigma$ of a completely positive map $\varphi : V_c \to W_c$ (see the line after Lemma \ref{duofs}) 
 is an important example of a real decomposable map.   
  To see that it is decomposable  note that since $\varphi$ is selfadjoint, we have $$\varphi(x)^* = \psi(x)^* - i \sigma(x)^* = \psi(x^*) + i \sigma(x^*) = \varphi(x^*),
 \qquad x \in V,$$
 so that  $\sigma(x)^* = -\sigma(x^*)$.   That is, $\sigma$ is antisymmetric, or {\em skew-adjoint}.  
 Also $(\psi,\psi) \in \mathfrak{P}(\sigma)$ by definition since $c(\psi,\sigma)$ is cp, and so 
 $\sigma$ is decomposable with $$\| \sigma \|_{\rm dec} \leq \| \psi(1) \| \leq \| \varphi (1) \| = \| \varphi \|.$$  
   In particular if $W = B(H) = M_n(\bR)$ then these `imaginary parts' are the maps   $\sigma
\in {\rm CB}(V,B(H))$  appearing in Lemma \ref{duofs}.  
We can also often arrange that $\sigma$ is (up to a positive scalar) the `imaginary part'  of a nc/matrix state.  Indeed, we shall explore this and some of the importance of this class of real decomposable maps in Section \ref{sec:DecAs}. 
\end{example}

The complex versions of the following string of results (see Chapter 6 of \cite{P}) are well-established in the literature for maps between $C^*$-algebras.
The usual proofs also work for maps between complex operator systems.   The real case can usually be proved using the same methods as in the complex case.  Instead, we derive them instantaneously from the complex results using the above, in particular 
\hyperref[res:RealInComplex]{Theorem \ref*{res:RealInComplex}}, illustrating how useful this result is.   

In what follows, $V$ and $W$ are real operator systems.

\begin{proposition}\label{prop:DecCompDec}
If $u \in \decR(V, W)$ and $u^\prime \in \decR(W, X)$ then $u^\prime \circ u \in \decR(V,X)$ and
$$
\|u^\prime \circ u\|_{\idec} \leq\|u^\prime\|_{\idec} \, \|u\|_{\idec}.
$$
\end{proposition}

\begin{proof}
The result follows from complexification and the corresponding complex result \cite[Proposition 6.6 (ii)]{P}. Let $u,u^\prime$ be as above. By \hyperref[res:RealInComplex]{Theorem \ref*{res:RealInComplex}}, we have $u^\prime \circ u \in \decR(V,X)$ since
 $u^\prime _c \circ u_c = (u^\prime \circ u)_c \in \decC(V_c,X_c)$, and
$$\|u^\prime \circ u\|_{\idec} = \|(u^\prime \circ u)_c\|_{\idec}\leq \|u^\prime _c\|_{\idec} \|u_c\|_{\idec} =\|u^\prime\|_{\idec} \|u\|_{\idec},$$ as desired. 
\end{proof}

\begin{proposition}\label{prop:verifyRuan2}
Let $A$ be a real $C^*$-algebra and $\alpha,\beta \in A.$ Define $\phi:A\to A$ by $\phi(x)= \alpha^* x \beta,$ then $\phi \in \decR(A,A)$ and $\|\phi\|_{\idec}\leq \|\alpha\|\|\beta\|.$
\end{proposition}

\begin{proof}
Follows immediately from complexification and the complex case (\cite[Proposition 6.10]{P}). Clearly $\phi_c(x)= \alpha^* x \beta$ and so  $\|\phi\|_{\idec} = \|\phi_c\|_{\idec} \leq  
\|\alpha\|\|\beta\|.$ \end{proof}

\begin{lemma}\label{lemm:DirectSumReal}
Let $V$ and $\left(W_{\iota}\right)_{{\iota} \in I}$ be real operator systems  and let $W=\oplus_{{\iota} \in I}^{\infty} \, W_{\iota}.$ Let $u: V \rightarrow W$. We denote $u_{\iota}= p_{\iota} u: V \rightarrow W_{\iota},$ 
where $p_{\iota}$ is the projection onto the ${\iota}$-th coordinate of $W$. Then $u \in \decR(V, W)$ if 
and only if $u_{\iota}\in \decR(V, W_{\iota})$ for all ${\iota}\in I$ with $\sup _{{\iota} \in I}\left\|u_{\iota}\right\|_{\idec}<\infty,$ and $\|u\|_{\idec} = \sup _{{\iota} \in I}\left\|u_i\right\|_{\idec}.$ 
\end{lemma}

\begin{proof}  The proof of \cite[Lemma 6.8]{P}
works to show  the complex case of the present result, for operator systems instead of $C^*$-algebras.
(One needs only note that the projections $p_{\iota}$ are cp contractions to prove the converse.) 
The real case is the same, or it follows by complexification: 
We may identify $W_c$ with $\oplus(W_{\iota})_c$. 
(See e.g.\ \cite[Section 5]{BReal}. In \cite[Lemma 5.14]{Sharma2014} the result is shown to hold for real operator spaces, clearly a similar proof holds for real operator systems.)  Then 
$(p_{\iota})_c = q_{\iota},$ where $q_{\iota}$ is the projection onto the ${\iota}$-th coordinate of $W_c.$ Indeed
$$(p_{\iota})_c( (x_{\iota}) + i (y_{\iota})) = x_{\iota} + i y_{\iota} 
= q_{\iota}((x_{\iota} + i y_{\iota})).$$
Hence $(u_{\iota})_c = (p_{\iota} u)_c = (q_{\iota} u_c).$  
We have $u\in \decR(V,W)$ if and only if  $u_c\in \decC(V_c,W_c)$.   Since $u_c$ may be identified with a map $V_c \to \oplus(W_{\iota})_c$, 
it follows from the complex case
that $u_c\in \decC(V_c,W_c)$ if and only if 
$q_{\iota} u_c = (p_{\iota} u)_c \in \decC(V_c,(W_{\iota})_c)$
for all ${\iota}\in I$ with $\sup _{{\iota} \in I}\left\|(u_{\iota})_c \right\|_{\idec} <\infty$.
That is, if and only if $u_{\iota}\in \decR(V, W_{\iota})$ for all ${\iota}\in I$, with $\sup _{{\iota} \in I}\left\|u_{\iota}\right\|_{\idec}<\infty$.  Indeed
 $$\|u\|_{\idec} =  \|u_c\|_{\idec}
= \sup _{{\iota} \in I}\left\|(u_{\iota})_c \right\|_{\idec}
= \sup _{{\iota} \in I}\left\| u_{\iota} \right\|_{\idec},$$ 
as desired.
\end{proof} 

 We identify $M_n(\decR(V,W))$ with $\decR(V,M_n(W))$  in the obvious way, 
thus equipping $\decR(V,W)$ with a series of matrix norms.  It is known that  $\decC(V,W)$ is an operator space for complex operator systems  $V$ and $W$ 
\cite{LM,Han}.

\begin{theorem} \label{thm:RealIsComplex_OS}
Let $V, W$ be operator systems. By identifying $M_n(\decC(V,W))$ with $\decC(V,M_n(W))$, we have that $\decC(V,W)$ is an operator space. Similarly $\decR(V,W)$ is a real operator space when $V,W$ are real operator systems. Moreover, in this case $\decR(V,W)_c \cong \decC(V_c,W_c)$. 
\end{theorem}

\begin{proof}
As in \cite[Theorem 6.1]{LM}, it is sufficient to check Ruan's axioms. Let $u\in \decR(V,M_n(W)), v\in \decR(V,M_m(W)$ and $\alpha,\beta \in M_n.$ By \hyperref[prop:verifyRuan2]{Proposition~\ref*{prop:verifyRuan2}} applied to  $x\mapsto \alpha u(x) \beta$ it is easy that 
$\|\alpha u \beta\|_{\idec} \leq \|u\|_{\idec} \|\alpha\| \|\beta\|.$
Similarly, 
 under the natural identification, \hyperref[lemm:DirectSumReal]{Lemma \ref*{lemm:DirectSumReal}} gives  $\|u\oplus v\|_{\idec} = \max\{\|u\|_{\idec},\|v\|_{\idec}\},$ so that $\decR(V,W)$ is an operator space.

The map $T \mapsto T_c$ is an isometry ${\rm Dec}_{\bR}(V,W)$ to $\decC(V_c,W_c).$
Indeed this is a complete isometry since $M_n(W)_c = M_n(W_c)$, and by appealing to the isometric case applied to  ${\rm Dec}_{\bR}(V,M_n(W))$. Since $\theta_V := x+iy\mapsto x -iy,\ x,y\in V$ and $\theta_W$ are real complete order isomorphisms, the proof of \cite[Theorem 2.3]{BCK} shows that $\decC(V_c,W_c)$ is a reasonable operator space complexification of $\decR(V,W)$. By Ruan's uniqueness result from \cite{RComp} it is {\em the unique} operator space complexification of $\decR(V,W),$ completing the proof.
\end{proof}

\begin{remark} \label{att}
In the complex case  it is known that the infimum in 
Definition \ref{def:DecR} is attained when $W$ is a complex von Neumann algebra, or more generally if $W$ is an operator system completely contractively complemented in $W^{**}$; see \cite[Remark 1.5]{Haag}, whose proof is the same even if the domain space is an operator system $V$.  Thus $\| u \|_{\rm dec} = \max \{ \| S_1 \| ,\| S_2 \| \}$ for some $(S_1, S_2) \in \mathfrak{P}(u)$. 
To see this is also all valid in the real case, let $W$ be a real von Neumann algebra for example, and $u\in \decR(V,W).$ Then $W_c$ is a complex von Neumann algebra, and $u_c \in \decC(V_c,W_c).$ 
 Hence by the above and Theorem  \ref{res:RealInComplex} and its proof we know that there exists $R_i \in \cpC(V_c,W_c)$ and $(S_1,S_2)\in \mathfrak{P}(u)$
such that 
$$ \|u_c\|_{\idec} \leq \max \{ \| S_1 \| ,\| S_2 \| \} \leq 
\max \{ \| R_1 \|, \| R_2 \| \} = \|u_c\|_{\idec} =
\|u\|_{\idec}.$$   Thus the infimum is attained.

If the infimum is attained, then the proof of Proposition \ref{ss} works to give that result even if $\|u\|_{\idec} = 1$. 
\end{remark}

\begin{proposition}\label{prop:DecIfuCP} 
For $u : V \to W$ the following properties hold:
\begin{enumerate}
\item [{\rm (1)}] If $u \in \cpR(V, W)$, then $\|u\|_{\idec}=\|u\|_{\rm cb}=\|u\|$.\label{enum:u_d=u_cb}
\item  [{\rm (2)}]  If $u$ is self-adjoint (i.e.\ $u=u^*$)  then $u \in \decR(V, W)$ if and only if $u \in {\rm Span}_{\bR} \, \cpR(V, W) = \cpR(V, W) - \cpR(V, W)$, and
in this case
$$
\|u\|_{\idec}=\inf \left\{\left\|u_1+u_2\right\| : u_1, u_2 \in \cpR(V, W),\ u=u_1-u_2\right\}. 
$$ 
 \end{enumerate}
\end{proposition}

\begin{proof} (1)\ If $u\in \cpR(V,W)$ then $u_c\in \cpC(V_c,W_c) \subseteq
\decC(V_c, W_c)$.  Thus 
$u \in \decR(V, W)$, and    $\|u\|_{\idec} = \|u_c\|_{\idec}  = \|u_c\|_{\icb} = \|u\|_{\icb}.$  

(2)\ If $u = u^*$ we have that $u_c = (u_c)^*$ (recall the proof of \hyperref[res:RealInComplex]{Theorem  \ref*{res:RealInComplex}}). Thus  $$\|u\|_{\idec} =\|u_c\|_{\idec} = \inf \{\|v_1 + v_2\|: v_1,v_2 \in \cpC(V_c,W_c),\ u_c = v_1-v_2\}.$$ 
This is dominated by $\inf \left\{\left\|u_1+u_2\right\| : u_1, u_2 \in \cpR(V, W),\ u=u_1-u_2\right\}$, since 
for such $u_1, u_2$ setting $v_i = (u_i)_c$ we have $u_c = v_1-v_2$ and  $$\|v_1 + v_2\| = \| v_1 + v_2\|_{\rm cb} = \| (u_1 + u_2)_c \| _{\rm cb} = \|u_1+u_2 \|_{\rm cb} = \|u_1+u_2 \|.$$ 
Conversely, if $u_c = v_1-v_2$ for $v_1, v_2
\in \cpC(V_c,W_c)$ then as in the last part of the proof of \hyperref[res:RealInComplex]{Theorem  \ref*{res:RealInComplex}} we have $u = \rho \circ (v_1-v_2)_{|V}$, 
and  $u_k = (\rho \circ v_k)_{|V} \in \cpR(V,W)$. 
We have $u = u_1 - u_2$ and 
 $\|u_1+u_2 \| =
  \| \rho \circ (v_1+v_2)_{|V} \| \leq \| v_1+v_2 \|.$   This together with 
 the inequality above 
proves the final equation. 
\end{proof}

{\bf Remarks.}  1)\ The result of course also follows from 
the proof in \cite[Lemma 6.5]{P}, using that $(u,u) \in 
\mathfrak{P}(u) \cap \mathfrak{P}(-u)$. 

\smallskip

2)\ In stark contrast to the complex theory, we do not in general have that $\decR(V,W) = {\rm Span}_{\mathbb{R}}(\cpR(V,W))$.
An example is the `imaginary part map' Im on $V = W = M_n(\bC)$, taking $[a_{ij}]$ to $[{\rm Im} \, a_{ij}]$.   
This map is completely contractive (see e.g.\  \cite[Lemma 2.6]{BMcI}).   
Note that $V$ is real injective by \cite{BCK}, and so $\decR(V,W)$ in this 
example is CB$_{\bR}(V)$ by Proposition \ref{prop:InjectiveImpliesCB=Dec}.  Every map in ${\rm Span}_{\mathbb{R}}(\cpR(V,W))$ is
selfadjoint, but Im is not selfadjoint.

\smallskip

3)\ If $W = B(H)$ then the infimum in (2) is achieved.
This is a well known result due to Wittstock in the complex case.  The real case may be deduced from 
this by applying the `real part projection' somewhat as in the proof of the converse of Theorem \ref{res:RealInComplex}.

\begin{proposition}\label{prop:InjectiveImpliesCB=Dec}
If $W=B_{\mathbb{R}}(H)$  or more generally if $W$ is a real injective $C^*$-algebra, then
$$
\decR(V, W)= \cbR(V, W),
$$
and for any $u \in \cbR(V, W)$ we have
 $\|u\|_{\idec}=\|u\|_{\icb}$.
\end{proposition}

\begin{proof}
If $u\in \cbR(V,W)$ then  $u_c \in \cbC(V_c,W_c) = \decC(V_c,W_c)$ by \cite[Proposition 6.7]{P}. Hence  $u\in \decR(V,W),$ by Theorem  \ref*{res:RealInComplex}. Furthermore, $$\|u\|_{\rm cb}=\|u_c\|_{\rm cb}=\|u_c\|_{\idec} = \|u\|_{\idec},$$ completing the proof.
\end{proof}

Henceforth we shall usually no longer reference Theorem  \ref*{res:RealInComplex} when it is obvious we are making use of it.

If $M$ is a real von Neumann algebra then $(M_*)_c = (M_c)_*$ as operator spaces. This follows immediately by uniqueness of the operator  space predual of a complex von Neumann algebra, and the fact that $(X^*)_c = (X_c)^*$ completely isometrically isomorphically for any operator space $X$, that $(M_*)_c$ is completely isometric isomorphic to  $(M_c)_*$.  Indeed both $(M_*)_c$ and  $(M_c)_*$ have operator  space dual $M_c$.   To see that $(M_*)_c = (M_c)_*$ as (matrix) ordered  spaces is more complex.
We recall that we discussed above Lemma \ref{duofs} the complexification of a (matrix) ordered space.  
This can be used to describe the matrix cones on 
$(M_*)_c$, and more generally, the cones for $(V_*)_c$ for a dual real operator system $V$.  
In the case that $V$ is a von Neumann algebra $M$, Lemma \ref{dos} below shows that $(M_*)_c \cong (M_c)_*$ as matrix ordered $*$-vector spaces.  We include the more general result because of its independent importance. 
We first recall some real operator space duality from \cite{BReal}:

\begin{lemma}\label{duos}  If $V$ is a real dual operator space with operator space predual $V_*$  then $(V_c)_*  \cong  (V_*)_c$ 
completely isometrically via the map 
$\gamma(\varphi ) = {\rm Re} \, \varphi_{|V}  + i \, 
{\rm Im} \, \varphi_{|V}$. 
\end{lemma}

   The following is the weak* version of \cite[Lemma 6.3]{BMcI}.  This is no doubt in the literature in the complex case, but since a reference does not 
come to hand we supply  a proof in the real case. 

    \begin{lemma}\label{CEun} Let $\varphi : V \to B(H)$ be a completely positive weak* continuous map on a real dual operator system.  Let $a = \varphi(1)^{\frac{1}{2}}$.  
 Then there exists weak* continuous ucp  $\Psi : V \to B(H)$ such that $\varphi = a \Psi(\cdot) a$.  \end{lemma}

                       \begin{proof}  We have  $\varphi = a \Psi(\cdot) a$ by \cite[Lemma 6.3]{BMcI}, we just need to argue that $\Psi$ may be taken to be 
            weak* continuous.   Suppose that $\overline{aH} = eH$ for projection $e$, and we have a bounded net $x_t \to x$ weak* in $V$.
            Then for $\zeta ,  \eta \in eH$ we have 
            $$\langle \Psi(x_t) a \zeta , a \eta \rangle = \langle \varphi(x_t)  \zeta ,  \eta \rangle \to \langle \varphi(x)  \zeta ,  \eta \rangle = 
            \langle \Psi(x) a \zeta , a \eta \rangle .$$ Since $aH$ is dense in $eH$, we have $e \Psi( \cdot) e$ weak* continuous by a standard Krein-Smulian argument.
            A careful look at the Choi-Effros argument cited in the proof of \cite[Lemma 6.3]{BMcI} now shows that $\Psi$ as defined there is 
            weak* continuous ($e \Psi( \cdot) e = \Psi( \cdot) e$, and $\Psi( \cdot) (1-e)$ is clearly weak* continuous).  \end{proof}  
            
                  If $V$ is a dual operator system then 
                  we write NCP$(V,B(H))$ for the 
                  weak* continuous (`normal') cp maps.  The `normal matrix/nc state space' need not be weak* closed.  However we still have: 

\begin{lemma}\label{dos}  If $V$ is a real operator system with operator space predual $V_*$ equipped with the canonical 
matrix cones, and if $(V_c)_*$ is  equipped with its canonical 
matrix cones, and  $n$ is a cardinal, then $\varphi \in   M_n((V_c)_*)^+ = {\rm NCP}(V_c,M_n(\bC))$ if and only if $c(\varphi) \in 
M_{2n}(V_*)^+ = {\rm NCP}(V,M_{2n}(\bR))$.  
Hence the canonical completely isometric isomorphism $(V_*)_c \cong (V_c)_*$ 
from  Lemma {\rm \ref{duos}} is also a complete order $*$-isomorphism as matrix ordered $*$-vector spaces. \end{lemma}

   \begin{proof} The first assertion follows from a simple variant of the proof of the `if and only if' in Lemma \ref{duofs}, using \cite[Lemma 5.2]{BReal}
   and Lemma \ref{CEun}.  The final assertion is a restatement of the first assertion, by the idea above Lemma {\rm \ref{duos}}: 
if (uniquely) $\varphi = \psi_c + i \sigma_c$ for $\psi, \sigma : V \to M_n(\bR)$ then 
 $c(\varphi)  = c(\psi, \sigma) \in 
 {\rm NCP}(V,M_{2n}(\bR))$ if and only if   $\varphi \in {\rm NCP}(V_c,M_n(\bC))$.  \end{proof}

 Note that the canonical projection $(V_c)^* = (V^*)_c \to V^*$ (resp.\ $(V_c)_* = (V_*)_c \to V_*$) takes            the canonical (matrix) positive cones onto the canonical (matrix) positive cones of $V^*$ (resp.\ $V_*$). 
            This gives another description of these cones.

 Let $u$ be a linear map between a real operator system   $V$ and a real von Neumann algebra $M.$ We define the map $$\tilde{u}: V^{**}\to M,\  \tilde{u}(x) := (u^{\star}|_{M_*})^{\star},$$ where $^\star$ is used to indicate the adjoint map. 

If  $A$ is a real $C^*$-algebra we may identify $(A_c)^{**}$ with $(A^{**})_c$ as $W^*$-algebras (see \cite[Lemma 5.2]{BReal} and \cite[Chapter 5]{Li}).
Similarly, $(V_c)^{**} \cong (V^{**})_c$ unitally 
complete order isomorphically as dual operator systems
\cite{BR} (see also Lemma \ref{duofs}).

\begin{lemma} \label{decdu}
Let $u: V \rightarrow M$ be a linear map from a real operator system $V$ to a real von Neumann algebra $M$. If $u \in$ $\decR(V, M)$ then $\tilde{u}  \in \decR(V^{* *}, M)$ and $\|\tilde{u} \|_{\idec}=\|u\|_{\idec}$.
\end{lemma}

\begin{proof}
Let $u\in \decR(V, M)$.  Then  $u_c \in \decC(V_c, M_c),$ which in turn gives 
$\stackrel{\resizebox{3 mm}{1mm}{$\sim$}}{u_c} \, 
\in \decC((V_c)^{**}, M_c)$ by the complex operator system case of \cite[Lemma 6.9]{P}.  Thus we are done, provided that $ \stackrel{\resizebox{3 mm}{1mm}{$\sim$}}{u_c} \,\in  \decC((V_c)^{**}, M_c)$ agrees with $(\tilde{u})_{c} \in  \decC((V^{**})_c, M_c)$,  under the above identification $(V^{{**}})_c \cong (V_c)^{**}$. We need only check that these maps restrict to the same map on $V^{**},$ and 
indeed only on $V$ since they are weak* continuous. However they both clearly restrict to $u$ on $V$. 
\end{proof}
\begin{proposition}
Let $C, A$ be real $C^*$-algebras. Let $u: C \rightarrow A^{**}$ and let $\left(u_i\right)$ be a net in the unit ball of $\decR(C, A)$ such that $u_i(x) \rightarrow u(x)$ with respect to $\sigma\left(A^{* *}, A^*\right)$ for any $x \in C$. Then $u \in \decR(C, A^{* *})$ with $\|u\|_{\idec} \leq 1$.
\end{proposition}

\begin{proof} Let $(u_i)$ and $u$ be as above. Of course $((u_i)_c)$ embeds into the space $\decC(C_c,(A^{**})_c)$, which we identify with $\decC(C_c,(A_c)^{**}).$  Then $(u_i)_c \to u_c$ with respect to $\sigma((A_c)^{**},(A_c)^{*})$ by \cite[Lemma 5.2]{BReal}.  So 
$u_c  \in {\rm Ball}(\decC(C_c,(A^{**})_c))$ by Lemma 8.24 in \cite{P}.   Thus $u \in \decR(C, A^{* *})$ with $\|u\|_{\idec} \leq 1$. 
\end{proof}

This can be used to verify that the proof of Theorem  8.25 in \cite{P}  works in the real case:

\begin{corollary} \label{nbid}  For any $n$ and  real $C^*$-algebra $A$ we have canonical isometries  $\decR(E,A^{**}) \cong \decR(E,A)^{**}$
if $E = M_n(\bR)$ or $E = \ell^\infty_n(\bR)$. 
\end{corollary}

\begin{example}  \label{P625} We will later mention the fact that for real linear $u : \ell^\infty_n(\bR) \to A$,  we have 
$$\| u \|_{\rm dec} = \inf \{\|  \sum_{k=1}^n \, a_k a_k^* \|^{\frac{1}{2}} \, \|  \sum_{k=1}^n \, b_k^* a_k \|^{\frac{1}{2}} : x_k = a_k b_k \; \text{for} \; a_k , b_k \in A \}.$$   
To see this by complexification 
note that $\| u \|_{\rm dec} = \| u_c \|_{\rm dec}$, which by  \cite[Lemma 6.25]{P} equals the same infimum expression here but over $a_k , b_k \in A_c$.  It is clear that 
this is dominated by the infimum with $a_k , b_k \in A$.  The converse inequality follows from the following considerations which we leave to the reader. Suppose that $$x_k = (a_k + i c_k)(b_k + i d_k) = 
a_k b_k - c_k d_k , 
\qquad a_k , b_k, c_k , d_k \in A.$$ Then 
 $$\| \sum_{k=1}^n \, a_k a_k^* + \sum_{k=1}^n \, c_k c_k^* 
\| \leq \| \sum_{k=1}^n \, (a_k + i c_k) (a_k + i c_k)^* \|,$$
and a similar inequality holds with $a_k$ and $c_k$ replaced by $b_k^*$ and $d_k^*$. 
\end{example} 

\smallskip

The relationship between real decomposability and complex decomposability is quite subtle.  
Clearly we cannot generally expect a 
real decomposable real linear map between complex systems to be complex decomposable.  A more subtle question, which we now proceed to answer, is if a
real decomposable complex map between complex operator systems is complex decomposable.
Certainly every complex decomposable is real  decomposable, with the same dec norm (exercise).  Moreover we have: 

\begin{lemma} \label{ncomma} Let $u : V \to B$ be a real decomposable map from a real operator system into a complex operator system   $B$.
Then $u$ extends uniquely to a complex linear decomposable  map $u' : V_c \to B$ with complex decomposable norm equal to 
$\| u \|_{\rm dec}.$
\end{lemma}

\begin{proof}    Let $\hat{u} : V \to M_2(B)$ be the usual completely positive map, with  $(S_1,S_2) \in \mathfrak{P}(u)$.  Then by \cite[Lemma 3.1]{BR}, 
$\hat{u}$ and $S_1, S_2$ 
extend uniquely to complex linear completely positive maps $\mu : V_c \to M_2(B)$ and $S_i' : V_c \to B$.  Also $u$
  extends uniquely to a complex linear  map $u' : V_c \to B$.  It is easy to see that the corners of $\mu$ are $S_i', u',$ and $(u')^*$.
  So $u'$ is decomposable, and $$\| u' \|_{\rm dec}
  \leq \max \{  \| S_1' \| , \| S_2' \| \}
  = \max \{  \| S_1 \| , \| S_2 \| \}.$$
  It follows that $\| u' \|_{\rm dec}
  \leq \| u \|_{\rm dec}.$  The converse follows since 
  the restriction of a decomposable map to a real subsystem is decomposable 
   (and also using  the exercise above the lemma), without increasing the Dec norm.
\end{proof} 

\begin{corollary} \label{or} A complex  linear map $u$ between complex operator systems is real decomposable if and only if it is complex decomposable, and in this case the real and complex decomposable norms of $u$ agree.
\end{corollary} 

\begin{proof} Suppose that $u : V \to B$ is a real decomposable complex map between complex operator systems
with real decomposable norm $t$.
By the lemma $u$ extends uniquely to a complex linear decomposable  map $u' : V_c \to B$ with 
$\| u' \|_{\rm dec} = t.$
Define $\nu : V \to V_c$ by $\nu(x) = \frac{1}{2} (x - \iota (i x))$.  Here we have written $\iota$ for the  ``$i$'' for the complex space $V_c$, to distinguish it from multiplication by $i$ in  the complex space $V$. 
It is easy to check that $u' \circ \nu = u$, and that 
$\nu$ is complex linear.  In fact $\nu$ is a completely positive isometry.   To see this note that under the 
identification of $V_c$ with $\cR_V$ below (\ref{cis}), 
$\nu$ corresponds to $x \mapsto x \otimes p$
where $p \in M_2(\bC)$ is the orthogonal projection 
$c(\frac{1}{2}(1+i))$.  Thus $\nu$ is decomposable with
$\| \nu \| = 1$, and so $u$ is complex decomposable with
complex decomposable norm dominated by 
$\| \nu \| \| u' \|_{\rm dec} = t$.   The converse inequality was left as an exercise above Lemma \ref{ncomma}. 
\end{proof} 

For its own independent interest, we discuss the bimodule Paulsen system, a canonical operator system constructed from a given operator $A$-$B$-bimodule.   The fact at the end of the discussion below may be used to provide a 
direct proof of Theorem \ref{thm:ExtRDec} below following the proof of \cite[Theorem 6.20]{P}.

All algebras are real unital $C^*$-algebras, and all operator modules are assumed to be nondegenerate for simplicity.  
Let $X$ be a real operator $A$-$B$-bimodule over real unital $C^*$-algebras $A$ and $B$ (see the discussion around \cite[Theorem 2.4]{BReal}). As in \cite[Chapter 3.6]{BLM} and \cite[Theorem 2.4]{BReal} we have a real CES (Christensen-Effros-Sinclair) representation.  That is, we may consider $A$ and $B$ as real unital-subalgebras of $B(H)$ and $B(K)$ respectively, and $X$ as an $A$-$B$-submodule of $B(K,H)$. We define the real bimodule Paulsen system as $${\mathcal S}(X) := \left\{\begin{bmatrix}
    a& x \\ y^* & b
\end{bmatrix} : a \in A, b \in B, x , y \in X\right\}.$$
(See \cite{Pnbook} or \cite{BLM} for the complex case, and \cite{BReal} for the real.) We may consider $X^\star\subset B(K,H)$ to be a $B$-$A$-bimodule. We view ${\mathcal S}(X)$ as a real unital selfadjoint $(A \oplus B)$-$(A \oplus B)$-submodule of $B(H \oplus K).$ Moreover, this operator system turns out to be independent of the CES representation chosen. 
As stated in \cite{BReal}, almost all the results in \cite[Chapter 3.6]{BLM} regarding the complex bimodule Paulsen system  (see also 
\cite{Pnbook}, these results are mostly due to Paulsen and his students) carry over to the real case.
We mention a couple of these.  In particular, let $X$ be as above and suppose we are given a completely contractive $A$-$B$-bimodule map $u: X \to B(K,H)$.  That is, $u(axb) = \theta(a) u(x) \pi(b)$ for $a \in A, b \in B, x\in X.$  Here $\theta: A\to B(H)$ and $\pi: B \to B(K)$ are unital $*$-homomorphisms. Then the real case of  3.6.1 in \cite{BLM} and ideas in the proof of \cite[Theorem 3.6.2]{BLM} tells us
$$\Theta  : \begin{bmatrix}
a & x \\
y^*   & b
\end{bmatrix} \mapsto 
\begin{bmatrix}
\theta(a) & u(x) \\
u(y)^*   & \pi(b)
\end{bmatrix}$$
is ucp as a map from ${\mathcal S}(X)$ into $B(H \oplus K)$.
Conversely, if $\Theta : {\mathcal S}(X) \to B(H \oplus K)$ is ucp then $\| u \|_{\icb} \leq  \|\Theta\|_{\icb} = \|\Theta(1)\|  = 1$. Thus we have that $\|u\|_{\icb} \leq 1$ if and only if $\Theta $ is cp. Indeed, we have:

\begin{proposition} If $u: X \to B(K,H)$ is real linear, and $\theta, \pi, \Theta$ are as above, then 
$\|u\|_{\icb} \leq 1$  and $u$ is an $A$-$B$-bimodule map if and only if $\Theta $ is cp.  
\end{proposition}

\begin{proof}  Indeed if $c = a \oplus b \in {\mathcal S}(X)$ then $$\Theta(c^* c) = \Theta \left( \begin{bmatrix}
a^*a & 0 \\
0   & b^* b 
\end{bmatrix} \right) = 
\begin{bmatrix}
\theta(a^* a) & 0 \\
0   & \pi(b^* b)
\end{bmatrix} = \Theta(c)^* \, \Theta(c).$$
Thus if we extend $\Theta$ to a ucp
map $C^*({\mathcal S}(X)) \to B(H \oplus K)$ it follows by Choi's multiplicative domain result
(the real case of \cite[Proposition 1.3.11]{BLM})
that 
$\Theta(c z c) = \Theta(c) \Theta(z) \Theta(c),$ where $z$ is the $2 \times 2$ matrix with zero entries except $x \in X$ in the $1$-$2$-corner.
It follows that $u$ is an $A$-$B$-bimodule map (this is a standard kind of argument). \end{proof}

\section{The real universal C*-algebra
and the \texorpdfstring{$\delta$}{delta}-norm} \label{sec:deltaNorm}

Pisier defines for a complex operator space $E$
the  universal $C^*$-algebra $C^*_{\bC} \langle E \rangle$ to be a $C^*$-algebra generated by a completely isometric copy of $E$, having the expected universal property
that any linear complete contraction $v : E
\to B(H)$ extends uniquely to a $*$-homomorphism $C^*_{\bC} \langle E \rangle \to B(H)$ \cite{Pisbk}.  (He also defines a unital version but we shall not say more on this here.)  As usual, this universal property defines $C^*_{\bC} \langle E \rangle \to B(H)$ uniquely up to 
$*$-isomorphism.  Exactly the same simple construction produces the universal real $C^*$-algebra $C^*_{\bR} \langle E \rangle$ of
a real operator space having the expected universal property.  Indeed one may define 
$C^*_{\bR} \langle E \rangle$ to be the 
real $C^*$-algebra $B$ generated by the canonical copy of $E$  in $C^*_{\bC} \langle E_c \rangle$.
We show that $B$ has the desired universal property.  Given a real linear complete contraction $v : E
\to B(H)$, the complex universal property
gives a complex 
$*$-homomorphism $\pi : C^*_{\bC} \langle E_c \rangle \to \bB$ extending $v_c$, where $\bB = B(H_c) \cong B(H)_c$.  Then 
$\pi(B)$ is contained in the copy of $B(H)$ inside $\bB$, and 
$\rho \circ \pi_{|B}$ does the trick,
where $\rho : B(H)_c \to B(H)$ is the canonical map.

\begin{lemma} \label{ucom} For a real operator space $E$ we have
   $C^*_{\bC} \langle E_c \rangle = (C^*_{\bR} \langle E_c \rangle)_c$.
\end{lemma}

\begin{proof}  We have already proved that $C^*_{\bR} \langle E_c \rangle$ is $*$-isomorphic to $B$,  the $C^*$-subalgebra of $C^*_{\bC} \langle E_c \rangle$
above.  The canonical period 2 unital completely isometric conjugate linear map $E_c \to E_c$ extends to a period 2 conjugate linear $*$-automorphism $\theta$ of $C^*_{\bC} \langle E_c \rangle$ that fixes $B$.  This follows by an almost identical argument for the analogous fact in the proof of \cite[Lemma 5.1]{BR}.  As in that result  we 
 are  done, since $B \cap iB = (0)$ and $B + iB = 
C^*_{\bC} \langle E_c \rangle$.  For completeness, we repeat the argument for the last two equalities.
Let $D$ be the fixed points of $\theta$, and $E$ the points with $\theta(x) = -x$.  These are closed with 
 $D \cap E = (0)$ and $D + E = C^*_{\bC} \langle E_c \rangle$. 
 Also $B \subseteq D$ and $iB \subseteq E$.  If $b_n, c_n \in B$ with 
 $b_n + ic_n \to x$.  By applying $\theta$ we see that 
 $b_n - ic_n \to \theta(x)$.  It follows that $x + \theta(x) \in B$ and $x - \theta(x) \in iB$, so that $x \in B + iB$.  Thus $B + iB$ is closed, hence is a complex C$^*$-subalgebra of
 $C^*_{\bC} \langle E_c \rangle$ containing 
 $E_c$.    Thus $B \oplus iB = 
C^*_{\bC} \langle E_c \rangle$.
\end{proof}

\begin{definition} \label{def:DeltaMap}
Given real operator spaces $E,F$ and linear maps $\theta: E \to B(H)$ and $\pi: F \to B(H),$ we define the map $\theta \cdot \pi: E \otimes F \to B(H)$ by $$\theta \cdot \pi\ \left(\sum x_j \otimes y_j \right) = \sum \theta(x_j) \, \pi(y_j),\ x_j\in E\ y_j\in F .$$
Let $A$ be a real $C^*$-algebra and 
$z \in E \otimes A$. We define  
$$\Delta(z) = \sup\{\| (\theta\cdot\pi) (z)\|_{B(H)} \},$$
where the supremum runs over all $(\theta, \pi, H)$ with $H$ is a real Hilbert space, $\pi: A \to B(H)$  a $*$-homomorphism, and $\theta : E \to \pi(A)^{\prime}$ a complete contraction.
\end{definition}

\begin{theorem}\label{thm:delta=max-real}
Let $E$ be a real operator space, and let $j:\ E \to C_r^*\langle E\rangle$ be the canonical completely isometric embedding into the real universal $C^*$-algebra $C_r^*\langle E\rangle.$ For any real $C^*$-algebra $A$ we have 
$$
\Delta(z)\ =\ \big\|(j\otimes I_A)(z)\|_{\,C_r^*\!\langle E\rangle\otimes_{\max} A},\ \qquad z \in E\otimes A.
$$
\end{theorem}

\begin{proof} 
For a real Hilbert space $H$, we have ${B}_{\mathbb{R}}({H})_{c} \cong {B}_{{\mathbb{C}}}({H}_c),$ and let us denote this map by $\eta.$ Suppose we are given $(\theta, \pi, H)$ as in Definition \ref{def:DeltaMap}. The complexification of $(\theta \cdot \pi)$ (as defined in Definition \ref{def:DeltaMap}) may be identified with  $(\eta \circ\theta_c) \cdot (\eta \circ \pi_c): E_c \otimes A_c \to B(H_c).$  Also
$(\eta \circ \theta_c, \eta \circ
\pi_c, H_c)$ satisfies the conditions in the last lines of the original (complex) definition of the $\Delta$ norm, the complex version of Definition \ref{def:DeltaMap} for
$\Delta_{\bC}$ on $E_c \otimes A_c$.  We have
$\| ((\eta \circ\theta_c) \cdot (\eta \circ \pi_c))(z) \| = \| (\theta \cdot \pi)(z) \|$ for $z \in E\otimes A$.
Therefore we have $\Delta(z) \geq \Delta_{\bC}(z)$. 

For the converse inequality,  we have that every complex Hilbert space $H$ is a real Hilbert space $H_r$ in the natural way. Suppose that we are  given $(\theta, \pi, H),$  with complex Hilbert space $H$,  complex $*$-homomorphism $\pi: A_c \to B(H)$ and complex complete contraction $\theta: E_c \to \pi(A_c)^\prime$.
Let $\pi' = \pi_{|A}$ and $\theta' = \theta_{|A}$, regarded as  real map
into $B_{\bR}(H_r)$.  One finds that $(\theta',\pi', H_r)$ obeys the prescription in Definition \ref{def:DeltaMap}, and 
 $\theta' \cdot \pi': E \otimes A \to B(H_r)$. 
  We have
$\| (\theta \cdot \pi)(z) \| = \| (\theta' \cdot \pi')(z) \|$ for $z \in E\otimes A$.  Hence $\Delta_{\bC}(z) \geq \Delta(z),$ and so $\Delta_{\bC}(z) = \Delta(z).$ 

 By e.g.\ \cite[Section 10.2]{BR} we have $\|z \|_{C_{\bR}^{*}\langle E \rangle \otimes_{\max} A} = \|z\|_{C^{*}_{\bC} \langle E_c \rangle \otimes_{\max} A_c}.$ 
 Therefore by the corresponding complex result \cite[Lemma 6.14]{P} we obtain $$\|z\|_{C_r^{*}\langle E \rangle \otimes_{\max} A} = \|z\|_{C^{*}\langle E_c \rangle \otimes_{\max} A_c} = \Delta_{\bC}(z) = \Delta(z)$$
 for $z \in E\otimes A$ as desired. 
\end{proof}

Next we investigate the real version of Pisier's $\delta$-norm in the complex case.   We write the latter, the original complex version, on $E_c \otimes A_c$, as $\delta_{\bC}$.

\begin{definition} \label{delta}
Let $A$ be a real unital $C^*$-algebra and $E$ a real operator space. For $z \in E\otimes A$ we let $$\delta(z)=\inf \left\{\|x\|_{M_n(E)}\left\|\sum a_i a_i^*\right\|^{1 / 2}\left\|\sum b_j^* b_j\right\|^{1 / 2}\right\}$$
where the infimum runs over all possible $n$ and all possible representations of $z$ of the form
$$
z=\sum_{i, j=1}^n x_{i j} \otimes a_i b_j,\ x= [x_{ij}]\in M_n(E), a_i, b_i \in A.
$$  
\end{definition}
 
\begin{theorem}\label{thm:DeltaVSdelta}
For $A$ a real unital $C^*$-algebra, and $E$ a real operator space, we have $\Delta(z)=\delta(z)$ for $z \in E\otimes A$. 
\end{theorem}

\begin{proof}
Consider $z \in E\otimes A$. Considering $z$ as a element of $E_c\otimes A_c,$ we have $\delta(z) \geq \delta_{\bR}(z),$ since an infimum over a bigger set, is smaller.  Hence  $\delta(z)\geq \delta_{\bC}(z) = \Delta_{\bC}(z) = \Delta(z),$ via \cite[Theorem 6.15]{P} and 
Theorem \ref{thm:delta=max-real}. Conversely $\Delta(z) \leq \delta(z)$ as in the proof of \cite[Theorem 6.15]{P}. Hence we are done. 
\end{proof}

The following is the real version of a result from \cite{JLM} (see 
\cite[Theorem 10.27]{P}).

\begin{proposition}  \label{jlm}
Let $u: B \rightarrow A$ be a bounded finite rank map between two real $C^*$-algebras. Then for any $\varepsilon>0$, there is an integer $n$ and a factorization $u=w \circ v$ of the form
$$
B \xrightarrow{v} M_n(\mathbb{R}) \xrightarrow{w} A
$$
with $\|v\|_{\icb} \|w\|_{\icb} \leq \|v\|_{\icb} \|w\|_{\idec} \leq\|u\|_{\idec}(1+\varepsilon)$.
Therefore, if $z \in B^* \otimes A$ is the tensor associated to $u: B \rightarrow A$, we have
$$
\|u\|_{\idec}=\delta(z) .
$$
\end{proposition}

\begin{proof} Fix $\epsilon$ and let $u$ be as in the hypothesis. Then $u_c: B_c\to A_c$ is a finite rank map. We invoke the corresponding complex theorem \cite[Theorem 10.27]{P}. Hence there exist complex (hence real, by Corollary 
\ref{or}) decomposable maps $v,w$ such that $u_c = w\circ v$ and $B_c \xrightarrow{v} M_n(\mathbb{C}) \xrightarrow{w} A_c ,$ 
and  $$\|v\|_{\icb} \, \|w\|_{\icb} \leq \|v\|_{\icb} \, \|w\|_{\idec} \leq \|u\|_{\idec}(1+ \epsilon).$$
Now $M_n$ is an injective $C^*$-algebra in both the real and complex case, so Dec agrees with CB on maps into $M_n$.
We may embed $M_n(\bC)$ in $M_{2n}(\bR)$ via the ucp map $c_n.$  Let $\rho_n : M_{2n}(\bR) \to M_n(\bC)$ be the left canonical left inverse of $c_n$.
Hence 
$$u = (\rho_B \circ w \circ \rho_n)\circ (c_n \circ v \circ \kappa_A).$$
Now $\rho_A \circ w \circ \rho_n,$ being the composition of real decomposable and ucp maps, is real decomposable with real decomposable norm dominated by the complex (=real)  decomposable norm  $\|w \|_{\idec}$,  using Corollary 
\ref{or}. Thus 
 $$\begin{aligned}
\|c_n \circ v \circ \kappa_B\|_{\icb} \,  \|\rho_A \circ w \circ \rho_n\|_{\idec} 
&\leq \|v\|_{\icb}\|w\|_{\idec}\\
&\leq \|u_c\|_{\idec} \, (1 + \epsilon)\\
&= \|u\|_{\idec} \,(1+ \epsilon).
\end{aligned}$$
For the second assertion, suppose that $z \in B^* \otimes A$ is the tensor associated to $u: B \rightarrow A$.  We regard 
$z \in (B^*)_c \otimes A_c$, which we may identify
with an element $w$ of $(B_c)^* \otimes A_c$.
It is easy to see that this is the tensor associated to $u_c: B_c \rightarrow A_c$. 
By the complex case we have  $$\|u\|_{\idec} = \|u_c\|_{\idec} = \delta_{\bC}(w) = \delta_{\bC}(z).$$ 
However $\delta(z) = \delta_{\bC}(z)$ by \hyperref[thm:DeltaVSdelta]{Theorem \ref*{thm:DeltaVSdelta}}.
\end{proof}

{\bf Remark.}  The `positive' analogue of this result, \cite[Lemma 10.15]{P} is valid in the real case with the same proof, as is the important Corollary 10.16 there.  This can be  used, as in the complex case, to prove the real case of the famous Choi-Effros-Kirchberg characterization of nuclearity, namely that a real $C^*$-algebra is nuclear (with the usual definition that the min and max tensor norms agree) if and only if it has the CPAP (see \cite[Theorem 10.17]{P}).  Another route to this is mentioned at the end of Section 11.1 in \cite{BR}, where CPAP is called the CPFP, namely via the approach by Kavruk \cite[End of Section 4.1]{Kavrukthesis}. 

\section{Characterization of \texorpdfstring{${\rm Dec}_{\rm as}$}{Dec as}} \label{sec:DecAs}

For real operator systems $V,W$ we have $\decR(V,W) = \decR(V,W)_{\rm sa} \oplus \decR(V,W)_{\rm as},$ 
where the latter is $\{u\in \decR(V,W): u^*= -u\}$.  
Indeed any  $u = (u+ u^*)/2 + (u- u^*)/2$, and it is easy 
to verify that $u \in \decR(V,W)$ iff $u^* \in \decR(V,W).$ This is immediate since 
 $U^* \, \hat{u}(\cdot) \, U$ is also  completely positive 
 where $U = \begin{bmatrix}
    0 & 1\\ 1 & 0
\end{bmatrix}$. Also,  
$$\decR(V,W)_{\rm sa}  = {\rm Span}_{\mathbb{R}}(\cpR(V,W)) = \cpR(V,W) - \cpR(V,W)$$ by the final equation  in Proposition \ref*{prop:DecIfuCP}. 
 We also saw after that Proposition that in stark contrast to the complex theory, we do not in general have  $\decR(V,W) = {\rm Span}_{\mathbb{R}}(\cpR(V,W))$.
 Therefore it is interesting to consider and characterize $\decR(V,W)_{\rm as} := \{u\in \decR(V,W): u^*= -u\}$.
 The complex variant of this question has an almost trivial solution.  If $X,Y$ are complex operator systems, then $$\decC(X,Y)_{\rm as} = i\decC(X,Y)_{\rm sa} = i(\cpC(X,Y) -  \cpC(X,Y)),$$  
(since $(i u)^*(x) = (iu(x^*))^* = - iu(x)$).
Unfortunately, the real case is completely different. 

Again let $V,W$ be real operator systems.   
In what follows we write $\rho_{W}$ and $\sigma_{W}$ for the canonical projection maps $W_c \to W$ defined by
$z = \rho(z) + i \sigma(z)$. 
Then  $$\cpR(V,W) =  \rho_W \, \cpC(V_c, W_c) \,\kappa_V = \rho_W \, \cpR(V, W_c),$$ where, naturally, $\rho_W \, \cpC(V_c, W_c) \, \kappa_V$ is defined to be $\{ \rho_W \circ v \circ   \kappa_V : v \in \cpC(V_c, W_c)\},$ with $\kappa_V$ the canonical inclusion in the complexification (i.e. $\kappa_V: V\to V_c:= x\mapsto x$). 
By analogy with the prescription above; for the real case, we define the class $\scp(V,W)$ of {\em skew-cp} maps, that is, $$\scp(V,W):= \sigma_W \, \cpC(V_c, W_c) \, \kappa_V = \sigma_W \, \cpR(V, W_c).$$ 
 We call these the `skew-cp' maps since they are a `skew' of the completely positive maps, and they are indeed  skew or antisymmetric: $u^*=-u.$  (We shall see that $\sigma$ is the `prototypical' skew-cp map, in the real case.) 
We check the last equality in the last centered equation: Every $u \in \cpR(V, W_c)$ gives 
$v \in \cpC(V_c, W_c)$ by $v = \rho_{W_c} \, \circ u_c$, and 
$$\sigma_W \circ v \circ  \kappa_V = \sigma_W \circ \rho_{W_c} \circ u_c \circ  \kappa_V = \sigma_W \circ \rho_{W_c} \circ  \kappa_{W_c}  \circ u 
= \sigma_W \circ u .$$ 
So $\sigma_W \cpR(V, W_c) \subseteq \sigma_W \cpC(V_c, W_c) \, \kappa_V$.   
Conversely, it is obvious that $\cpC(V_c, W_c) \kappa_V  \subseteq \cpR(V, W_c)$.
Thus $$\scp(V,W) =  \{ \beta : V \to W |\  \exists \alpha \in \cpR(V,W) \;  s.t.\ \alpha + i \beta \in \cpR(V, W_c) \}.$$

We  define $\decR(V,W)^\circ_{\rm as}$ to be the set of
maps in $\decR(V,W)_{\rm as}$ whose 1-1 and 2-2 corners agree, that is, $$\decR(V,W)^\circ_{\rm as} := \{u\in \decR(V,W)_{\rm as} : \ \exists S\in \cpR(V,W): (S,S) \in \mathfrak{P}(u)\}.$$ 
Finally, we define the {\em imaginary completely positive
maps} ${\rm ICP}(V,W)$ to be the `imaginary parts' of 
completely positive
maps $\varphi : V_c \to W_c$, as defined in the  line after Lemma \ref{duofs}.

\begin{theorem} \label{decas} For real operator systems $V,W$ we have:
\begin{enumerate}
\item   [{\rm (1)}]  $\decR(V,W)_{\rm as} = \decR(V,W)^\circ_{\rm as} = \scp(V,W) = {\rm ICP}(V,W).$ 
\item  [{\rm (2)}] If $W$ is completely contractively complemented in $W^{**}$ (e.g.\ if $W$ is a dual operator system or real von Neumann algebra) then for any $u : V \to W$ in the set(s) in {\rm (1)} there exist
cp maps $S : V \to W$ and $\varphi : V_c \to W_c$ such that 
$u$ (resp.\ $S$) is the imaginary (resp.\ real) part of $\varphi$, $\hat{u} = c(S, u) = c(\varphi)$ (notation  as in {\rm (\ref{cis})}),
and $\| u \|_{\rm dec} = \| S \|$. 
\item  [{\rm (3)}] Under the conditions of {\rm (2)} we may also further choose  $S$ 
to be a positive scalar (namely,  $\| u \|_{\rm dec}$) multiple of a ucp map, and 
$$\| \hat{u} \| = \| \varphi \| = \| u \|_{\rm dec} + \| u(1) \|.$$  In particular if $W = B(H)$ and $u(1) = 0$ then we may choose $S, \hat{u}$ and $\varphi$ 
to each be a nc/matrix state multiplied by $\| u \|_{\rm dec}$. 
\end{enumerate} 
\end{theorem} 

\begin{proof}  (1)\ The ideas are mainly as in Example \ref{exrd}.  We first show that $\scp(V,W) \subseteq \decR(V,W)_{\rm as}$. Suppose that $u \in \scp(V,W)$, with $u= \sigma_W \circ v \circ  \kappa_V $ for some map $v$ in $\cpC(V_c, W_c)$. Note that  $$\sigma_W((a + ib)^*) = \sigma_W( a^* - i b^*) = -b^* = - \sigma_W(a + ib)^* , \qquad a, b \in V.$$  Hence $\sigma_W$ is skew, and thus so is $u = \sigma_W \circ v \circ  \kappa_V$.  Similarly $\sigma_W$ is decomposable, and therefore  so is $u = \sigma_W \, \circ \, v \, \circ  \, \kappa_V$, since $v$ and $\kappa_V$ are 
real cp.  To see that $\sigma_W : W_c \to W$ is decomposable, note that 
$$\begin{bmatrix}
\rho & -\sigma_W \\ \sigma_W & \rho 
\end{bmatrix} = I_{W_c}$$ may be viewed as the (ucp) identity map of $W_c$.

If $\beta \in \scp(V,W)$ with $\alpha + i \beta \in \cpR(V,W_c)$, then we may identify the latter with $$\hat{\beta} = \begin{bmatrix}
    \alpha & -\beta \\
    \beta & \alpha
    \end{bmatrix} ,$$
which is therefore completely positive.  So $\beta \in 
{\rm Dec}^\circ(V,W)_{\rm as}$. 
This argument is reversible, hence $\scp(V,W) = \decR(V,W)^\circ_{\rm as}.$

If $u \in {\rm Dec}(V,W)_{\rm as}$ with  
$$\hat{u} = \begin{bmatrix}
    S_1& u\\
    -u &S_2
\end{bmatrix}$$ completely positive, 
then by a shuffle and unitary multiplication we also have 
 $$\begin{bmatrix}
    S_2 & u\\
    -u & S_1
\end{bmatrix}$$ completely positive. Since the average of these two matrices is completely positive, we see that $u \in \scp(V,W)$ and in $\decR(V,W)^\circ_{\rm as}$. Hence $$\decR(V,W)_{\rm as} = \decR(V,W)^\circ_{\rm as} = \scp(V,W).$$
We leave it to the reader that $\scp(V,W) = {\rm ICP}(V,W).$ 

(2)\ If $W$ is completely contractively complemented in $W^{**}$ then by Remark \ref{att} the Dec norm is achieved.  That is, there exist $(S_1, S_2) \in \mathfrak{P}(u)$ with $\max \{ \| S_1 \| ,\| S_2 \| \} = \| u \|_{\rm dec}.$  
By the argument above, we may replace both $S_1$ and $S_2$ with their average $S$,
and in the last displayed equation we may replace the max with $\|S \|$.    Also, $\hat{u} = c(\varphi)$   by Lemma \ref{duofs} for a cp map $\varphi \in \cpC(A_c, B(H)_c)$ with
$\| \varphi \| = \| \hat{u} \|$, and we have 
$u= \sigma_{B(H)} \circ \varphi \circ  \kappa_V$, the `imaginary part'.  Similarly $S$ is the `real part' of $\varphi$. 

(3)\ Finally, 
by the last assertion in  Remark \ref{att} we can take $S$  to be $\| u \|_{\rm dec}$ times a ucp
map. 
It is a simple exercise (using spectral theory or a direct calculation), that if 
$x$ is a real or complex operator with $x^*=-x$ then $\| c(I,x) \| = 1 + \| x \|$.
Thus we have $$\| \varphi \| = \| \hat{u} \| = 
\| \hat{u}(1) \| = \| c(\| u \|_{\rm dec} 1, u(1)) \| = \| u \|_{\rm dec} + \| u(1) \|.$$
Thus $u(1) = 0$ if and only if 
 $\| \hat{u} \| = 
 \| u \|_{\rm dec}$, and in this case it is  easy to see that $\hat{u}$ and $\varphi$ are ucp.  
\end{proof}

{\bf Remarks.} 1)\  There is a natural norm on 
 $\decR(V,W)_{\rm as} = \decR(V,W)^\circ_{\rm as}$ defined by $\inf \{ \| S \|: (S,S)\in \mathfrak{P}(u)\}$, which is easily seen to equal the
$\idec$ norm by the averaging trick above. 

\smallskip

2)\ By the same reasoning as in 1), even if $W$ is not 
contractively complemented in $W^{**}$, we can find 
cp maps $S : V \to W$ and $\varphi : V_c \to W_c$ satisfying all of the conditions in the last statements of the theorem, except that 
some of the numbers  are `within epsilon'.  E.g.\ 
$\| S \|$ is within $\epsilon$ of   $\| u \|_{\rm dec}$ in (2), and in (3)  the occurrences of 
$\| u \|_{\rm dec}$ should be replaced by `a number within epsilon of' $\| u \|_{\rm dec}$. 

\smallskip

3)\ Actually it is not known whether the contractively complemented assumption in the theorem is necessary at all.   Indeed this is related to the apparently still open question in \cite{Haag} as to whether the infimum in $\| u \|_{\rm dec}$ is always achieved.   Perhaps  Theorem \ref{decas} can be used to  give insight into this problem. 

\medskip 

We now relate the above to nc convexity as mentioned in the introduction: Let $K$ be a real nc convex set in the sense of \cite{BMcI}.   The complexification $K_c$ of $K$ consists of elements $z= x + iy$ where $c(x,y) \in K$.  We call $x$  the `real part' and $y$ the `imaginary part'.  We call the set of such $y$ the `imaginary part' of $K_c$.  This is also a nc convex set \cite{BMcI}.

If $K$ is also compact then we may assume by the duality in that paper that 
$K$ is the nc state space ${\rm ncS}(V)$ of a real  operator system $V$.  That is, 
$K_n = {\rm UCP}(V,M_n(\bR))$ for each $n$.  Let $W = B(H) = M_n(\bR)$ for a cardinal $n$.  Let $C$ be the 
real nc convex set with $C_n = {\rm CP}(V,M_n(\bR))$ for each $n$.
Next we consider SCP$(V,W)$ as a real nc convex set (and operator space).
We do this by identifying it with the nc subset ${\rm CB}(V,W)_{\rm as}$ of the 
operator space ${\rm CB}(V,W)$.
Write $\bB$ for the real compact nc convex set
of matrix unit balls for  ${\rm CB}(V,W)_{\rm as} = {\rm Dec}(V,W)_{\rm as}$, these  with its canonical operator space structure.  In particular $\bB_1 = {\rm Ball}({\rm CB}(V,W)_{\rm as}) = {\rm Ball}(\scp(V,W))$.  Write $\bB_0$ for the real nc convex compact subset of $\bB$ consisting of 
the maps $u$ with $u(1) = 0$. 

\begin{theorem} \label{coco} We have that $\bB_0$ is the `imaginary part' 
of the complexification $K_c$ of the real nc compact convex set $K$ above. 
Also,  ${\rm SCP}(V,W)$ is the imaginary part 
of the complexification $C_c$ of the real nc convex set $C$ above.
\end{theorem} 

\begin{proof} If $x + iy \in K_c$ then 
$x$ is in $K = {\rm ncS}(V)$, and can be identified with a ucp map $S : V \to M_n(\bR)$.   The `imaginary part'
$y$ may be viewed as a real linear map $u : V \to M_n(\bR)$.
We know from \cite{BMcI} that $K_c$ can be identified with 
ncS$(V_c)$.   Let $\varphi = S_c + i u_c : V_c \to M_n(\bC)$. Then $u$ is a map in the set in (1) of the previous theorem with $W = M_n(\bR)$, and all of the conditions in (2) and (3) of the theorem are met.  See also  Lemma \ref{duofs}.
In particular 
$\hat{u}$ is ucp and $u \in \bB_0$ (that $u(1)=0$ is immediate from $\hat{u}(1) = I$).  
 Conversely, if $u \in \bB_0$, so that 
$u : V \to M_n(\bR)$  is an antisymmetric complete contraction with $u(1) = 0$, then by (3) we may choose  $S, \hat{u}, \varphi$ 
to be nc/matrix states.  Indeed $u$ and $S$ are the real and imaginary part of 
the complex nc state $\varphi$ of $V_c$.
That is,  $S + i u \in 
K_c = {\rm ncS}(V_c)$. The other assertion is similar but easier. 
\end{proof}  

  It seems remarkable that every 
map in $\bB_0$ is the imaginary part of a ucp map.
Of course the situation above is somewhat simplified because ${\rm CB}(V,W)_{\rm as} = {\rm Dec}(V,W)_{\rm as}$
when $W = M_n(\bR)$.  However in a more general nc convexity situation one might have a  more general $W$ where ${\rm CB}(V,W) \neq {\rm Dec}(V,W)$, but the dec norm is still attained (and we recall that it is open as to whether the latter is always true).  In this case, the last theorem would still be valid, relating the antisymmetric decomposable maps to 
the `imaginary part’ of a complex nc convex set.

\medskip 

{\bf Example.}  For the quaternions $\bH$ we have by  Proposition \ref{prop:InjectiveImpliesCB=Dec} that 
Dec$_{\rm sa}(\bH,\bH) = {\rm CB}_{\rm sa}(\bH,\bH)$ since $\bH$ is injective \cite{ROnr}.
This is 10 (real) dimensional, while Dec$_{\rm as}(\bH,\bH) = {\rm CB}_{\rm as}(\bH,\bH)$ is 6 dimensional.  Indeed, every linear map
$v : \bH_{\rm as} \to \bH_{\rm as}$ defines 
an element in Dec$_{\rm sa}(\bH,\bH)$ by $x \mapsto 
v(x - x_1 1)$.  Here $x_1$ is the coefficient of $1$ in $x$.
(We recall that any linear map between finite dimensional operator spaces, or more generally any bounded finite rank map, is completely bounded.  This follows because bounded functionals are
completely bounded.) This gives 9 of the dimensions, while the last comes from the map $x \mapsto x_1 1$.

Similarly, every 
$y \in \bH_{\rm as}$ defines 
an element in Dec$_{\rm as}(\bH,\bH)$ by $x \mapsto 
x y + y x$.    This gives 3 of the dimensions, while the other 3 come from the maps $x \mapsto x_1 \, y$.

\bigskip

We suggest an investigation of the theory of $\scp$ maps.  This should hopefully be useful for example in the study of real nc convexity, for the reasons given in our `imaginary part' discussion in the introduction.   There should be many aspects of the theory of completely positive maps that one may generalize to SCP maps.  The following  results are now very simple, however one could make a long list of known results about completely positive maps and ask if there are `decomposable variants'.

\begin{theorem}[Stinespring dilation theorem for SCP maps] \label{Sti}
Let $A$ be a real unital $C^*$-algebra and let $\phi \in \scp(A,B(H)).$  Then there exist a complex Hilbert space $K$ and a unital $*$-homorphism $\pi: A_c \to B(K),$ and an operator $T : H_c \to K$, with  $\|\phi \|_{\rm dec} + \| \phi(1) \| =  \|\phi \|_{\rm cb} + \| \phi(1) \| =\| T \|^2$, such that $$\phi(a) = \sigma(T^*\pi(a)T) , \qquad a \in A.$$ 
Here $\sigma = \sigma_{B(H)} : B(H)_c = B(H_c) \to B(H)$ is the canonical projection (taking $R_1 + i R_2$ to $R_2$ for $R_i \in B(H)$).
\end{theorem}

\begin{proof}  By Proposition \ref{prop:InjectiveImpliesCB=Dec} we have Dec$(A,B(H)) = {\rm CB}(A,B(H))$ isometrically. Moreover the `dec norm' is achieved,  and by  Theorem \ref{decas} there is a  cp map $v \in \cpC(A_c, B(H)_c)$ with 
$\phi = \sigma_{B(H)} \circ v \circ  \kappa_V$ and $\| \phi \|_{\rm cb} + \| \phi (1) \| =  \| v \|$. 
By Stinespring's theorem
there exists a complex Hilbert space $K$ and a unital $*$-homomorphism $\pi: A_c \to B(K),$ and an operator $T : H_c \to K$ with $\| v \| = \| T \|^2$ and 
 $v(a) = T^*\pi(a)T$ for $a \in A$.  Thus 
$$\phi(a) = \sigma(T^*\pi(a)T), \qquad a \in A,$$ as desired.  
\end{proof} 

{\bf Remark.} Conversely a map of the form in the last theorem is SCP, by the definition of the latter.  

\bigskip

Arveson's extension theorem for SCP maps is on the other hand obvious: 
Let $A$ be a real $C^*$-algebra and $V$ be an operator subsystem of $A.$ Given $\phi \in \scp(V,B(H)),$ with $H$ a real Hilbert space, then there exists some $\psi\in \scp(A,B(H))$ extending $\phi.$  Moreover this may be done with $\| \phi \|_{\rm cb} = \| \psi \|_{\rm cb}$. 
Indeed if $u$ is a complete contraction in $\scp(V,B(H))$
then there exists 
a 
completely contractive extension 
$v : A \to B(H)$.   Then $\frac{1}{2}(v - v^*)$ does the trick.

Combining this Arveson extension with our Stinespring version above 
gives a generalized version of Theorem \ref{Sti}.  This represents  a map $\phi \in \scp(V,B(H))$ as $$\phi(x) = \sigma(T^*\pi(x)T) , \qquad x \in V,$$
for a unital $*$-representation $\pi$ of a complexification of any  unital $C^*$-algebra containing $V$ as an operator subsystem, with  $\|\phi \|_{\rm cb} + \| \phi(1) \| =\| T \|^2$. 

\section{Other applications of Dec} \label{sec:Apps}

In this section, we check the real case of several results from \cite{P,Pisbk} and \cite{Haag} related to decomposable maps.

\subsection{Max-tensorizing maps}

\begin{proposition}\label{cor:ProbChild}
Let $V,W,X$ be real operator systems. For any $u \in \decR(V,W)$ and all $x\in X\otimes V$ we have 
$$
\left\|\left(I_X \otimes u\right)(x)\right\|_{X \otimes_{\max} W} \leq\|u\|_{\idec} \, \|x\|_{X \otimes_{\max } V}.
$$

Moreover, the mapping $I_X \otimes u: X \otimes_{\max } V \rightarrow X \otimes_{\max } W$ is decomposable and its norm satisfies

$$
\left\|I_X \otimes u\right\|_{{\decR}(X \otimes_{\max } V, X \otimes_{\max } W)} \leq\|u\|_{\idec}.
$$
\end{proposition}

\begin{proof}
The proof follows from complexification and the complex case of the hypothesis \cite[Proposition 6.11]{P}. For any real operator systems $A,B;$ we may identify $A_c\otimes_{\max} B_c$ with $(A\otimes_{\max}B)_c,$ as complex operator systems, see \cite[Theorem 10.9]{BR}. Under this identification if $S:A\to X$ and $T:B\to Y$ then the  linear  map $(S\otimes T)_c: (A\otimes_{\max}B)_c \to (X\otimes_{\max}Y)_c$ is identified with $S_c\otimes T_c: A_c\otimes_{\max} B_c\to X_c\otimes_{\max} Y_c,$ 
 To see this note that we identify $X_c\otimes Y_c$ with $(X\otimes Y)_c$ 
 and check that they agree on their respective restrictions to $X\otimes Y,$ that is they restrict to $S\otimes T.$

Let $u\in \decR(V,W),$ and consider the map $I_X \otimes u.$ Then its complexification $(I_X\otimes u)_c,$ may be identified with $(I_X)_c\otimes u_c.$ As stated above, we know that $(I_X)_c \otimes u_c$ is decomposable, and hence $I_X\otimes u$ is decomposable. Since any operator system embeds unitally completely isometrically into its complexification; by invoking the complex version of the hypothesis \cite[Proposition 6.11]{P}, we have for any $ x\in X \otimes V \subseteq X_c \otimes V_c$ that
$$
\begin{aligned}
 \|(I_X \otimes u)(x)\|_{X\otimes_{\max} W}=\left\|((I_X)_c \otimes u_c)(x)\right\|_{X_c \otimes_{\max} W_c} &\leq\|u_c\|_{\rm dec} \, \|x\|_{X_c \otimes_{\max } V_c}\\
 &= \|u\|_{\rm dec} \,  \|x\|_{X \otimes_{\max} V}.
\end{aligned}
$$

Subject to the same identifications as in the discussion above, we also have
$$\begin{aligned}
\|I_X\otimes u\|_{{\rm Dec}(X \otimes_{\max} V, X\otimes_{\max} W) } &= \|(I_X)_c\otimes u_c\|_{{\rm Dec}(X_c \otimes_{\max} V_c, X_c\otimes_{\max} W_c) }\\
&\leq \|u_c\|_{\idec}= \|u\|_{\idec}.
\end{aligned}$$
This completes the proof. 
\end{proof}

The following is then immediate, as usual: 

\begin{corollary}
Let $u_j \in \decR(A_j, B_j),\ j=1,2$ be decomposable mappings between real operator systems. Then $u_1 \otimes u_2$ extends to a decomposable mapping in $\decR(A_1 \otimes_{\max } A_2, B_1 \otimes_{\max } B_2)$ such that
$$
\begin{aligned}
\left\|u_1 \otimes u_2\right\|_{\decR\left(A_1 \otimes_{\max } A_2, B_1 \otimes_{\max } B_2\right)} \leq\left\|u_1\right\|_{\idec}\left\|u_2\right\|_{\idec}.
\end{aligned}
$$
\end{corollary}

\begin{proposition}
Let $u \in \decR(V, W)$ be a finite rank map between real operator spaces. For any real operator system $X$ and $x\in X\otimes V$ we have
$$
\left\|\left(I d_X\otimes u\right)(x)\right\|_{X \otimes_{\max } W} \leq\|u\|_{\idec} \, \|x\|_{X \otimes_{\min } V}.
$$
\end{proposition}

\begin{proof}
Complexification and \cite[Proposition 6.13]{P} yield the desired outcome. Indeed, let $u\in \decR(V,W)$ be of finite rank, then $u_c \in \decR(V_c,W_c)$ is finite rank. Therefore as in the last proof, for $x \in X \otimes V$ we have
 $$\left\|\left(I d_X\otimes u\right)(x)\right\|_{X \otimes_{\max } W} = \|(I_{X_c} \otimes u_c)(x)\|_{X_c \otimes_{\max} W_c}\leq \|u_c\|_{\idec} \, \|x\|_{X_c\otimes_{\min}V_c}.$$
However $\|u_c\|_{\idec} \, \|x\|_{X_c\otimes_{\min}V_c} = \|u\|_{\idec} \, \|x\|_{X \otimes_{\min } V}$.  
This holds because 
 indeed  $X_c \otimes_{\min} V_c \cong (X \otimes_{\min} V)_c$ as complex operator systems \cite[Corollary 10.4]{BR}. 
 \end{proof}

\subsection{Tensorial  applications of real decomposable maps}

We next check the real case of a result used many times in \cite{P}.  In \cite{P} a result less general than the one ending the discussion at the end of Section \ref{sec:RealDecMaps} is used to prove the complex case of the next result, and a similar  argument works in the real case.  We give a quick alternative proof by complexification.

For $C^*$-algebras $A$ and $D$, if $E \subseteq A$ we will write 
$D \otimes_m E$ for the completion of the operator space 
structure inherited on $D \otimes E$ via the inclusion $D \otimes E \subseteq D \maxten A$.

\begin{theorem}\label{thm:ExtRDec}
Let $A$ be a unital $C^*$-algebra, $E \subseteq A$ an operator space and $M \subseteq B(H)$ a von Neumann algebra. Let $u: E \rightarrow M$ be a bounded linear map. Let $\hat{u}: M^{\prime} \otimes E \rightarrow B(H)$ be defined by $$\hat{u}\left(x^{\prime} \otimes x\right)=x^{\prime} \, u(x)\qquad   x^{\prime} \in M^{\prime},\ x \in E.$$ Then $\hat{u}$ extends to a completely bounded map on $M^{\prime} \otimes_{m} E$, which we still write as $\hat{u}$, with $\| \hat{u}: M^{\prime} \otimes_{m} E \rightarrow$ $B(H) \|_{\rm cb} \leq 1$ if and only if there is $\tilde{u} \in \decR(A, M)$ with $\|\tilde{u} \|_{\text {dec }} \leq 1$ extending $u$. In other words
$$
\left\|\hat{u}: M^{\prime} \otimes_{m} E \rightarrow B(H)\right\|_{c b}=\inf \left\{\|\tilde{u}: A \rightarrow M\|_{\idec} : 
\tilde{u}_{\mid E}=u\right\}, 
$$
and the infimum is attained.
\end{theorem}

\begin{proof} For any $X \subseteq A$ and $M$ as above, and any bounded linear map $u: A\to M$ let $\hat{u}$ be defined as in the hypothesis. Then $(u_c)^{\hat{}} = (\hat{u})_c.$  To see this we first recall that $(M')_c = (M_c)'$. Indeed viewing $M_c = M+iM \subseteq B(H) + i B(H)$, clearly 
$(M')_c = M' +i M' \subseteq  (M_c)'$. Conversely, if $T_1 + i T_2 \in (M_c)' \subseteq M'$ then $T_i \in M'$.  Thus 
$T_1 + i T_2 \in (M')_c = M' +i M'$.  So  $(M')_c = (M_c)'$.  

The algebraic (vector space) complexification of 
$M^{\prime} \otimes E$  is $(M')_c \otimes_{\bC} E_c
= (M_c)' \otimes E_c$ (suppressing the field subscript). Thus the complexification of 
$\hat{u}$ takes $x'_1 \otimes y_1 + i x'_2 \otimes y_2$ to $x'_1 u( y_1) + i x'_2 u( y_2)$, for $x_k' \in M', y_k \in E$.  However this agrees with  $\widehat{u_c}$  applied to the 
canonical copy of $x'_1 \otimes y_1 + i x'_2 \otimes y_2$  in $(M_c)' \otimes E_c$.  Thus 
$\widehat{u_c}$ may be identified with $(\hat{u})_c$. Since $$M' \maxten A \subseteq (M' \maxten A)_c =  (M')_c \maxten A_c = 
(M_c)' \maxten A_c$$ (see \cite[Section 10.2]{BR}),
it follows that $$M' \otimes_{m} E 
\subseteq (M')_c \otimes_{m} E_c \subseteq
(M')_c \maxten A_c$$ completely isometrically, and 
$(M' \otimes_{m} E)_c = (M')_c \otimes_{m} E_c$. 

Let $\tilde{u} \in \decR(A,M)$ be an extension of $u.$ Then $\tilde{u}_c \in \decC(A_c,M_c)$ is an extension of $u_c.$ Hence by the complex case in \cite[Theorem 6.20]{P} the result follows.  Indeed $\|\widehat{u_c} \|_{\icb} \leq 1,$ but $\widehat{u_c} = (\hat{u})_c.$ Hence $1\geq \|\widehat{u_c} \|_{\icb} = \|(\hat{u})_c\|_{\icb}= \|\hat{u} \|_{\icb}.$

Conversely, assume that $\|\hat{u} \|_{\icb}\leq 1.$ Then $\|(\hat{u})_c\|_{\icb} \leq 1.$ Since $\widehat{u_c} = (\hat{u})_c$ we have $\|\widehat{u_c} \|_{\icb} \leq 1.$ By assumption we have that there exists a map $ \stackrel{\resizebox{3 mm}{1mm}{$\sim$}}{u_c} \, \in \decC(A_c,M_c)$ with $\| \stackrel{\resizebox{3 mm}{1mm}{$\sim$}}{u_c} \|_{\idec} \leq 1$ and extending $u_c.$ Then 
$\rho_M \circ  \stackrel{\resizebox{3 mm}{1mm}{$\sim$}}{u_c}_{|A} \, \in {\rm Ball}(\decR(A,M))$ (c.f.\ the  last part of the proof of Proposition \ref{prop:DecIfuCP}), and this 
map extends $u$. 
\end{proof}

The real case of Kirchberg's theorem on decomposable maps (\cite[Theorem 14.1]{Pisbk}):

\begin{corollary} \label{cop} Let $A, B$ be real $C^*$-algebras and $u : A \to B$. 
Then $i_{B} \circ u$ is a contraction in ${\rm Dec}(A,B^{**})$
if and only if $I_D \otimes u : D \maxten A \to D \maxten B$ is a contraction for any $C^*$-algebra $D$. 
\end{corollary} 

\begin{proof}  If  $i_{B} \circ u$ is Dec-contractive then so is its 
complexification $(i_{B} \circ u)_c = i_{B_c} \circ u_c : A_c \to 
(B_c)^{**} \cong (B^{**})_c$.     By the complex case,
$I_{D_c} \otimes u_c : D_c \maxten A_c \to D_c \maxten B_c$ is a contraction for any $C^*$-algebra $D$.   Restricting, 
$I_D \otimes u : D \maxten A \to D \maxten B$ is a contraction.

For the converse, if we replace $D$ by $M_n(D)$,
it is easy to see that $I_D \otimes u : D \maxten A \to D \maxten B$ is a complete contraction.
Hence it extends to a complex complete contraction 
between complexifications
$$(I_D \otimes u)_c = I_{D_c} \otimes u_c
: D_c \maxten A_c \to D_c \maxten B_c.$$
For any complex $C^*$-algebra $C$, using the noncanonical complex sequence $C \to C_c \to C$
we obtain a complex linear contraction $v : 
C \maxten A_c \to C \maxten B_c$
as the composition $$C \maxten A_c \to C_c \maxten A_c \to C_c \maxten B_c \to C \maxten B_c.$$ Note that $v$  takes $$c \otimes x \to (c,\bar{c}) \otimes x \to (c,\bar{c}) \otimes u_c(x) \to c\otimes u_c(x),
\qquad c \in C, x \in A_c.$$
That is,  $v = I_C \otimes u_c$ is a contraction. 
Thus by the complex case, $i_{B_c} \circ u_c : A_c = (i_{B} \circ u)_c \to 
(B_c)^{**} \cong (B^{**})_c$ is a decomposable contraction. 
Hence so is its restriction to $A$.
\end{proof} 

The real case of  \cite[Corollary 14.6]{Pisbk} 
(see also \cite[Theorems 7.4, 7.6]{P}) is similar.  We recall that 
$D \otimes_m X$ is the operator space 
structure inherited on $D \otimes X$ via the inclusion $D \otimes X \subseteq D \maxten A$. This result is related to Theorem \ref{thm:ExtRDec}, but we have chosen to prove both by complexification since the techniques we use for this may be helpful elsewhere.

\begin{corollary} Let $A, B$ be real unital $C^*$-algebras, $X$ 
a subsystem of $A$, and let $u : X  \to B$ be linear.  
Then $u$ has an  extension in ${\rm Ball}(\decR(A,B^{**}))$
if and only if $I_D \otimes u : D \otimes_m X \to D \maxten B$ is a contraction for any $C^*$-algebra $D$. 
\end{corollary} 

\begin{proof}  If  $u$ has a 
decomposable contractive extension $v \in {\rm Ball}(\decR(A,B^{**}))$ then $u_c$ has a 
decomposable contractive extension $v_c \in {\rm Ball}(\decR(A_c,B_c^{**}))$.    By the complex case \cite[Theorem 7.4]{P},
$I_{D_c} \otimes u_c : D_c \otimes X_c \subseteq D_c \maxten A_c \to D_c \maxten B_c$ is a contraction for any $C^*$-algebra $D$.   Restricting, 
$I_D \otimes u : D \otimes_m X  \to D \maxten B$ is a contraction.

The converse is mostly identical to the proof of 
Corollary \ref{cop}: if we replace $D$ by $M_n(D)$,
then $I_D \otimes u : D \otimes_m X  \to D \maxten B$ is a complete contraction which 
 extends to a complex complete contraction 
$$(I_D \otimes u)_c = I_{D_c} \otimes u_c
: (D \otimes X)_c = D_c \otimes_m X_c \to D_c \maxten B_c.$$
And for any complex $C^*$-algebra $C$,
we obtain a complex linear contraction $v = I_C \otimes u_c: 
C \otimes_m X_c \subseteq C \maxten A_c \to C \maxten B_c$. 
Thus by the complex case, $u_c$ has an extension 
$v :  A_c  \to 
(B_c)^{**} \cong (B^{**})_c$ which is a decomposable contraction. 
Hence $(\rho_B)^{\star \star} \circ v_{|A}$ is a decomposable contraction into $B^{**}$, and 
$(\rho_B)^{\star \star}(v(x)) = (\rho_B)^{\star \star}(u(x)) = u(x)$ for $x \in X$. 
\end{proof}

{\bf Remarks.} 1)\ Similarly the results in  7.10 and Theorem  10.14 in \cite{P} seem to work the same (with the same proof) in the real case, but using our real versions of some basic results on  decomposable maps used in  Theorem 10.14.  

\medskip

2)\ Similarly, Theorem 7.29--7.30  in \cite{P} is also valid in the real case, with essentially the same proof but using our real versions of some basic results on  decomposable maps used in \cite[Theorem 7.29--7.30]{P}.  The biggest issue one encounters here is that the proof that (i)' implies (i) in Theorem 7.29 uses \cite[Corollary 7.16]{P} which in turn uses the fact
that any complex unital $C^*$-algebra is a quotient of $C^*(F)$ for a free group $F$, which is not true in the real case.  To circumvent this, we prove that 
(i)' implies (i)  by complexification.
Using facts in \cite[Section 10.2]{BR} the canonical $*$-monomorphism in (i)' complexifies to 
 the canonical $*$-monomorphism $\mathcal{C} \maxten D_c \to 
\mathcal{C} \maxten A_c$.  Thus by the complex case
we get the  canonical $*$-monomorphism $B_c \maxten D_c \to 
B_c \maxten A_c$ for any $C^*$-algebra $B$.
Restricting, we get the  canonical $*$-monomorphism $B \maxten D \to B \maxten A$.  That is, we have (i).
The proof also  uses \cite[Corollary 7.27]{P}, but this was checked in the real case in \cite{BR}.

\subsection{Dec characterization of injective von Neumann algebra and of WEP} \label{secIWEP}

\begin{theorem} \label{hinj} Let $M$ be a real von Neumann algebra. Then $M$ is injective if and only if  there is a constant $c > 0$ such that $\| u \|_{\rm dec} \leq c \| u \|_{\rm cb}$ for all
$n \in \bN$ and all linear $u : l^\infty_n \to M$.  \end{theorem} 

\begin{proof}  Suppose that there exists such a constant.
 Since the complexification of 
${\rm CB}_{\bR}(A,M)$ is ${\rm CB}_{\bC}(A_c,M_c)$ 
(see \cite[Theorem 2.3]{BCK}, every complete contraction $u \in {\rm CB}_{\bC}(l^\infty_n(\bC),M_c)$
is of form $v_c + i w_c$ for complete contractions $v, w \in 
 {\rm CB}_{\bR}(l^\infty_n(\bR),M)$.
 Thus $\| u \|_{\rm dec} \leq \| v \|_{\rm dec} + \| w \|_{\rm dec} \leq 2 c$. 
  Haagerup showed in \cite{Haag} that  this implies $M_c$ is injective, hence so is $M$ by \cite{BCK}. 
 
 The converse holds with $c = 1$ by Proposition \ref{prop:InjectiveImpliesCB=Dec}.  \end{proof}

 Haagerup proved in the complex case that a $C^*$-algebra $A$ has WEP if and only if  $\decR(l^\infty_n(\bC),A) = \cbR(l^\infty_n,A)$  isometrically for all
 $n \in \bN$ (see \cite{P} for Pisier's proof of this).  We do not yet have the real case of this.
 We present some partial results, and a discussion of this point.

\begin{corollary}
A real $C^*$-algebra $A$ has WEP if and only if  $\decR(l^\infty_3,A) = \cbR(l^\infty_3,A)$ completely isometrically. 
\end{corollary} 

\begin{proof}  The complex case is due to Junge  and Le Merdy \cite[Proposition 3.5]{JLM}.   Their proof shows that a complex $C^*$-algebra $B$ has WEP if and only if
$\| u \|_{\rm cb} =\| u \|_{\rm dec}$ for all $n \in \bN$ and $u : l^\infty_3 \to M_n(B)$.
Suppose that $u : l^\infty_3(\bR) \to M_n(A)$ for a real $C^*$-algebra $A$.
If $A$ has WEP then so does $M_n(A)$, and so $M_n(A)_c =
M_n(A_c)$ has complex WEP \cite{BR}.
Thus 
$$\| u \|_{\rm cb} = \| u_c \|_{\rm cb} =\| u_c \|_{\rm dec} =\| u \|_{\rm dec}.$$
The converse is similar.
If $\decR(l^\infty_3,A) = \cbR(l^\infty_3,A)$ completely isometrically, then by complexifying, we see that $${\rm Dec}_{\bC}(l^\infty_3(\bC),A_c) = \cbC(l^\infty_3(\bC),A_c).$$  So $A_c$ has WEP by the complex case, and hence so has $A$ by e.g.\ \cite[Proposition 4.1]{RComp}.  
\end{proof}

The following is the real case of part of Theorem 23.2 in \cite{P},
and has the same proof but using the real versions checked earlier of some basic results on decomposable maps.

\begin{proposition} \label{P232}
Let $B$ be a $C^*$-algebra. Let $\iota: A \rightarrow B$ be the inclusion mapping from a $C^*$-subalgebra $A \subseteq B$. The following are equivalent: \begin{enumerate}
\item [{\rm (i)}]  For any $n \geq 1$ and any $u: \ell_{\infty}^n \rightarrow A$ we have
$$
\|u\|_{\decR\left(\ell_{\infty}^n, A\right)}=\|\iota \circ u\|_{\decR\left(\ell_{\infty}^n, B\right)} .
$$
\item [{\rm (ii)}]  For any $n \geq 1$ and any $v: \ell_{\infty}^n \rightarrow A^{* *}$ we have
$$
\|v\|_{\decR\left(\ell_{\infty}^n, A^{* *}\right)}=\left\|\iota^{\star \star} \circ v\right\|_{\decR\left(\ell_{\infty}^n, B^{* *}\right)}
$$
\end{enumerate}
\end{proposition}

\begin{corollary} If $A$ is a 
real or complex $C^*$-algebra, then 
$A$ is nuclear if and only $A$ is  locally reflexive and there is a constant $c > 0$ such that $\| u \|_{\rm dec} \leq c \| u \|_{\rm cb}$ for all
$n$ and all linear $u : l^\infty_n \to A$.   
\end{corollary} 

\begin{proof} 
Clearly if $A$ is nuclear then $A$ is  locally reflexive (one may see this by passing to the complexification and appealing to \cite[Proposition 4.3]{RComp}), and has WEP \cite{BR}.  Thus 
the condition involving the constant $c$ holds with $c = 1$, just as in the complex case \cite[Corollary 23.5]{P}. 
Indeed if $A \subset B(H)$ and $P : B(H)^{**} \to A^{**}$ is a contractive cp projection and $i^{\star \star} : A^{**} \to B(H)^{**}$ is the inclusion then $\| u \|_{\rm dec} = \| i^{\star \star} \circ u \|_{\rm dec}$ by Proposition \ref{prop:DecCompDec}.   So $\| u \|_{\rm dec} = \| u \|_{\rm cb}$ by Proposition \ref{P232}. 
Conversely, suppose that the condition involving the constant $c$ holds.  
Suppose that $v :  \ell^\infty_n \to A^{**}$ is completely contractive.
 We may approximate  $v$ point weak* if $A$ is locally reflexive
 by a net of completely contractive $v_t$ in $\cbR(\ell^\infty_n , A)$
 by the real case of \cite[Proposition 8.28]{P}.
 The  $v_t$  have dec norm $\leq C$, and $\decR(\ell^\infty_n , A^{**})$ is the  bidual space of 
 $\decR(\ell^\infty_n , A)$ by the real case of \cite[Theorem 8.25]{P}. So the net has a weak* convergent subnet
 with dec norm $\leq C$, and its limit must be $v$.
 Thus  $\| v \|_{\rm dec} \leq C$.  Thus $A^{**}$ is injective by Theorem \ref{hinj},
 so that $A$ is nuclear (again one may see this by passing to the complexification, appealing to
 facts in \cite{BR}).  \end{proof}

{\bf Remark.} 
In the last corollary, the last condition alone ought to actually characterize when a $C^*$-algebra has WEP, but this seems open even in the complex case (Gilles Pisier indicated that this was the case on an  email enquiry in early 2025). 

The real versions of the important results 23.2--23.5 in \cite{P} are probably true, but attempting to follow parts of the complex proofs there seems to be temporarily blocked by gaps  in 
the existing real theory of von Neumann algebras.
For example, we do not know if there is a nice decomposition 
of  real  von Neumann algebras similar to the fact that a complex von Neumann algebra is 
a direct sum of von Neumann algebras of the form 
$B(H) \otimes \mathcal{N}$ where $\mathcal{N}$ is a $\sigma$-finite von Neumann algebra?  This is probably not a difficult question, and it is an important one. 

For a $C^*$-algebra $A$ and linear $u : l^\infty_n \to A$, if  $u(e_k) = x_k$ for $k = 1, \cdots , n,$ then we have $$\| u \|_{\rm cb} = \| u_c\|_{\rm cb} 
= \| \sum_{k=1}^n \,
 U_k \otimes x_k \|_{C^*_{\bC}(F_n) \otimes_{\rm min} A_c} = \| \sum_{k=1}^n \,
 U_k \otimes x_k \|_{C^*_{\bR}(F_n) \otimes_{\rm min} A},$$ 
 where in the third equality we used (3.7) from  \cite{P}. Here the $U_k$  are the 
free generators of the full free group $C^*$-algebra. 
 On the other hand, by \cite[Lemma 6.28]{P} we have 
$$\| u \|_{\rm dec} = \| u_c \|_{\rm dec} = \| \sum_k \, U_k \otimes u_c(e_k) \|_{C^*_{\bC}(F) \maxten A_c}.$$
However the latter equals $\| \sum_k \, U_k \otimes u(e_k) \|_{C^*_{\bR}(F) \maxten A}$,  since it is shown in  
\cite{BR} that 
$(C^*_{\bR}(F) \maxten A)_c = C^*_{\bC}(F) \maxten A_c$. 
 Indeed $\| u \|_{\rm dec}$ is a supremum of $\| \sum_{k=1}^n \, u_k \pi(x_k) \|$ over all representations
 $\pi : A \to B(H)$ and all unitaries $u_i \in \pi(A)'$. (See the proof of \cite[Lemma 6.28]{P}; the real case of some of the ingredients there
 are in our Section \ref{sec:deltaNorm}.)  
Also $$\| u \|_{\rm dec} = \inf \{\|  \sum_{k=1}^n \, a_k a_k^* \|^{\frac{1}{2}} \, \|  \sum_{k=1}^n \, b_k^* b_k \|^{\frac{1}{2}} : x_k = a_k b_k \}$$ by Example  \ref{P625}. 
We can thus rephrase the question above in terms of whether
WEP is characterized by the minimal and maximal norms agreeing on $\cS_n \otimes A$, where $\cS_n$ is the span of the 
free generators $U_k$.  This is related to results of Kavruk \cite{Kavrukthesis} (with Paulsen and other collaborators), however we do not have these equalities at the  matrix levels as they do.  

We expect that there is also a real version of Haagerup's other  deep characterization of WEP in \cite[Theorem 23.7 (i)--(iii)]{P}.   The equivalence of (i) and (iv) there does hold:

\begin{theorem} A 
real $C^*$-algebra $A$ has WEP if and only if the canonical map $A \to A^{**}$ factors completely boundedly through $B(H)$ for some $H$.  A real $C^*$-subalgebra $A$ of $B(H)$ has WEP if it is
completely boundedly complemented in $B(H)$.  \end{theorem}

\begin{proof} These follow easily 
by complexification and the complex case of these statements, using the fact from 
\cite{BR} that $A$ has WEP if and only if $A_c$ does.
Indeed standard arguments used many times in \cite{BR} show that if these `completely boundedly conditions' hold for $A$ then they hold for $A^{**}$. 
The second assertion in the last theorem is the real case of \cite[Corollary 23.10]{P}. \end{proof}

\subsection{Dec characterization of QWEP}\label{sec:QWEP}
The Dec characterization of QWEP in \cite[Theorem 9.67]{P} works in the real case at least for $C^*$-algebras that are densely spanned by unitaries, such as all von Neumann algebras.  Indeed  (i) implies (ii) in \cite[Theorem 9.67]{P} in the real case, for all real $C^*$-algebras.  This  follows by complexification.  A direct proof of this needs 7.48 and 6.11 in \cite{P}, which were mostly checked in \cite{BR}.  As in the complex case
it is obvious that (ii) implies (ii)' there, and that if 
$D$ is densely spanned by unitaries, then these imply (iii)
there.  The converse was proved in \cite{BR}.

We can give a variant of the condition above to take care of $C^*$-algebras that are not densely spanned by unitaries:

\begin{corollary} \label{qweps} A real C$^*$-algebra $D$ is real QWEP  if and only if for all real LLP $C^*$-algebras $C_1$ and $C$ and real decomposable maps $u : C \to M_2(D)$ with $\| u \|_{\rm dec} \leq 1$, we have $I \otimes u$ contractive from $C_1 \minten C \to C_1 \maxten M_2(D)$. 
\end{corollary}

\begin{proof} If $D$ is real QWEP, with $A = B/I$ for WEP $B$, then $M_n(D) = M_n(B)/M_n(I)$ and $M_n(B)$ is WEP. 
So $M_2(D)$ is QWEP.  Hence the condition holds by the last proof.  

Conversely, suppose that the condition holds.    Now let  $u : C \to D_c$ be a complex  decomposable map,
viewed as a map into $M_2(D)$.   We have $I \otimes u$ contractive from $C_1 \minten C \to C_1 \maxten M_2(D)$.  Composing with
$I \otimes \Phi$, where $\Phi : M_2(D) \to D_c$ is the canonical projection, we have $I \otimes u$ contractive from $C_1 \minten C \to C_1 \maxten D_c$.
So $D_c$ has QWEP hence $D$ has QWEP.
\end{proof}

{\bf Remark.} Corollaries 9.69 and 9.70 in \cite{P} are valid in the real case, by complexification, as are 9.71--9.75.  We leave these to the reader.

\bigskip 

{\bf Acknowledgements.}   
  We acknowledge support from NSF Grant DMS-2154903.  We thank Christian Le Merdy for very helpful discussions and historical  comments, and for thinking about some difficult questions which we asked him. 
    We also thank Ping Zhong for discussions on an aspect of this.

\bigskip 

{\bf Data availability:}
No data was used for the research described in the article.

\end{document}